\documentclass[12pt]{amsart}
\usepackage{amssymb}
\usepackage{eucal}
\usepackage{bm}
\usepackage[all,cmtip]{xy}
\usepackage{color}
\usepackage[bookmarks=false]{hyperref}
\usepackage{multirow,bigdelim}
\usepackage[small,nohug,heads=vee]{diagrams}
\usepackage{tikz}
\usetikzlibrary{matrix,arrows}
\usepackage{commath}
\usepackage{enumerate, tikz-cd}

\allowdisplaybreaks

\newtheorem{theorem}[equation]{Theorem}
\newtheorem{lemma}[equation]{Lemma}
\newtheorem{proposition}[equation]{Proposition}
\newtheorem{corollary}[equation]{Corollary}

\theoremstyle{definition}
\newtheorem{definition}[equation]{Definition}

\theoremstyle{remark}
\newtheorem{remark}[equation]{Remark}

\numberwithin{equation}{subsection}

\allowdisplaybreaks[1]

\newcommand{\ds}{\displaystyle}

\newcommand{\fitt}{\textrm{Fitt}}

\def\XXint#1#2#3{{\setbox0=\hbox{$#1{#2#3}{\int}$ }
\vcenter{\hbox{$#2#3$ }}\kern-.6\wd0}}

\DeclareMathOperator{\mspec}{MSpec}

\newcommand{\Fq}{\mathbb{F}_{q}}

\newcommand{\cM}{\mathcal{M}}

\newcommand{\cO}{\mathcal{O}}

\newcommand{\cK}{\mathcal{K}}
\newcommand{\cC}{\mathcal{C}}
\newcommand{\cR}{\mathcal{R}}

\newcommand{\cF}{\mathcal{F}}

\DeclareMathAlphabet{\matheur}{U}{eur}{m}{n}

\newcommand{\fm}{\mathfrak{m}}

 \DeclareMathOperator{\Gal}{Gal}

\DeclareMathOperator{\Vol}{Vol}
\DeclareMathOperator{\Fitt}{Fitt}

\newcommand{\OX}{\mathcal{O}}

\newcommand{\power}[2]{{#1 [\![ #2 ]\!]}}
\newcommand{\laurent}[2]{{#1 (\!( #2 )\!)}}

\definecolor{ForestGreen}{rgb}{0.0, 0.5, 0.0}

\newcommand{\C}{\ensuremath \mathbb{C}}
\newcommand{\Z}{\ensuremath \mathbb{Z}}

\newcommand{\F}{\ensuremath \mathbb{F}}

\newcommand{\cS}{\mathcal{S}}
\newcommand{\cU}{\mathcal{U}}

\newcommand{\fp}{\mathfrak{p}}

\newcommand{\inv}{\ensuremath ^{-1}}

\newcommand{\CO}{\mathcal{O}}

\makeatletter
\@namedef{subjclassname@2010}{\emph{2020} Mathematics Subject Classification}
\makeatother

\def\XXint#1#2#3{{\setbox0=\hbox{$#1{#2#3}{\int}$ }
\vcenter{\hbox{$#2#3$ }}\kern-.6\wd0}}

\title[An Equivariant Tamagawa Number Formula for Drinfeld Modules]{An Equivariant Tamagawa Number Formula \linebreak for Drinfeld Modules and Applications}

\author{Joseph Ferrara}

\author{Nathan Green}

\author{Zach Higgins}

\author{Cristian D. Popescu}
\address{Dept. of Mathematics, University of California at San Diego, San Diego, CA 92093, USA}
\email{jferrara@ucsd.edu}
\email{n2green@ucsd.edu}
\email{zhiggins@ucsd.edu}
\email{cpopescu@math.ucsd.edu}

\keywords{Drinfeld Modules, Motivic $L$--functions, Equivariant Tamagawa Number Formula, Brumer--Stark Conjecture}

\subjclass[2010]{11G09, 11M38, 11F80}

\thanks{The first and third author were partially supported by NSF RTGrant DMS-1502651}

\date{}

\begin{document}

\begin{abstract} We fix motivic data $(K/F, E)$ consisting of a Galois extension $K/F$ of characteristic $p$ global fields with arbitrary abelian Galois group $G$ and a Drinfeld module $E$ defined over a certain Dedekind subring of $F$. For this data, we  define a
$G$--equivariant motivic $L$--function $\Theta_{K/F}^E$ and prove an equivariant Tamagawa number formula for appropriate Euler product completions of its special value $\Theta_{K/F}^E(0)$.
This generalizes to an equivariant setting the class number formula proved by Taelman in 2012 for the value $\zeta_F^E(0)$ of the Goss zeta function $\zeta_F^E$ associated to the pair $(F, E)$. (See also Mornev's 2018 work for a generalization in a very different, non--equivariant direction.) Taelman's result is obtained from ours by setting $K=F$.
As a notable consequence, we prove a perfect Drinfeld module analogue of the classical (number field) refined Brumer--Stark conjecture, relating a certain $G$--Fitting ideal
of Taelman's class group $H(E/K)$ to the special value $\Theta_{K/F}^E(0)$ in question.
\end{abstract}

\keywords{Drinfeld Modules, Motivic $L$--functions, Equivariant Tamagawa Number Formula, Brumer--Stark Conjecture}
\date{\today}
\maketitle

\setcounter{tocdepth}{1}
\tableofcontents

\section{Introduction}

 In \cite{T12} Taelman proved a beautiful class number formula for the special value at $s=0$ of the $\Bbb C_\infty$--valued Goss zeta function $\zeta_F^E(s)$ associated to a characteristic $p$ global field $F$ and a Drinfeld module $E$ defined over a certain Dedekind domain  $\cO_F\subseteq F$.
 Since Taelman's formula establishes an equality between the special value $\zeta_F^E(0)$ and a quotient of (what we interpret below as) volumes of two compact topological groups canonically associated to the pair $(F, E)$, it can be naturally viewed as a ``Tamagawa number formula'' for the pair.

In this paper we consider a pair $(K/F, E)$, where $K/F$ is a Galois extension of characteristic $p$ global fields of abelian Galois group $G$ and a Drinfeld module $E$ defined over the ring $\cO_F$.
To the pair $(K/F, E)$ we associate a $G$--equivariant version $\Theta_{K/F}^E(s)$ of the Goss zeta function, which takes values in the group ring $\Bbb C_\infty[G]$ associated to $G$.
We extend and refine Taelman's techniques to this $G$--equivariant setting and prove an equality between appropriate Euler-completed versions of $\Theta_{K/F}^E(0)$ and the quotient of the $G$--equivariant volumes of (properly modified versions of) Taelman's topological groups, now endowed with a natural $G$--action. We obtain in this way a $G$-equivariant Tamagawa number formula in this setting, generalizing Taelman's class number formula.
(See Theorem \ref{T:ETNF}.) A generalization of Taelman's formula in a very different, non--equivariant direction, was obtained by Mornev in \cite{Mornev}.

Further, we use our main result to prove a Drinfeld module analogue of the far reaching classical (number field) refined Brumer--Stark conjecture, which relates the Euler--completed versions of the special value $\Theta_{K/F}^E(0)$ to a certain $G$--Fitting ideal of Taelman's class group $H(E/K).$ (See Theorem \ref{T:BrSt}.) While important in its own right, this result also opens the door to developing an Iwasawa theory for Taelman's class group--like invariants and their generalizations arising from a similar study of the special values $\Theta_{K/F}^E(n)$, with $n\in\Bbb Z_{\geq 1}$.  We should also mention that particular cases of these special values feature prominently in log-algebraicity theorems, as in \cite{CEGP18}, for example.

Unlike previous approaches to a $G$--equivariant theory (see \cite{AT15} and \cite{Fang15}), we consider a general abelian group $G$, whose order is allowed to be divisible by the characteristic prime $p$. This makes it impossible for us to work on a character-by-character basis (as there are no $p$--power order characters of $G$ with values in characteristic $p$) and therefore the ensuing theory is truly $G$--equivariant. In addition, this creates serious $G$--cohomological obstructions, mostly due to the presence of wildly ramified primes in $K/F$. These obstructions justify our need to construct certain taming modules $\cM$ for $K/F$, which lead to the appropriate Euler--completed versions $\Theta_{K/F}^{E, \cM}(0)$ of $\Theta_{K/F}^E(0)$.

In this introduction, we define more precisely the arithmetic data $(K/F, E)$, construct its associated Galois equivariant $L$--function
$\Theta^E_{K/F}$ and its Euler--completed versions, give an infinite product formula for their special values at $s=0$, describe the relevant class of $G$--equivariant compact topological groups, briefly describe a $G$--equivariant volume  function on this class, and state the main results of this paper.

Our suggestion to the reader is to read the Introduction, followed by the Appendix (where several algebraic tools, mostly of homological nature, are developed), then read sections 2--6 in their natural order. Section 6 contains the proofs of the main theorems.

\subsection{The arithmetic data $(K/F, E)$}\label{S:intro-data} In what follows, $p$ is a fixed prime number, $q$ is a fixed power of $p$, and $\F_q(t)$ is the rational function field in one variable $t$ over the finite field of $q$ elements $\F_q$. For any characteristic $p$ field $K$, we denote by $\overline K$ its separable closure and by $G_K:={\rm Gal}(\overline K/K)$ its absolute Galois group.  For any commutative $\F_q$--algebra $R$, we denote by $\tau:=\tau_q$ the $q$--power Frobenius of $R$, i.e. the $\F_q$--algebra endomorphism $\tau:R\to R$ sending $x\to x^q$.  As usual $R\{\tau\}$ denotes the twisted polynomial ring in $\tau$, with relations
$$\tau\cdot x=x^q\cdot\tau, \text{  for all }x\in R.$$
The ring $R\{\tau\}$ is the largest $\F_q$--subalgebra of the endomorphism ring ${\rm End}_R(\Bbb G_a)=R\{\tau_p\}$ for the affine line $\Bbb G_a$, viewed as a group scheme over ${\rm Spec}(R).$
\medskip

Let $F$ be a finite, separable extension of $\F_q(t)$ and let $K$ be a finite abelian extension of $F$,  such that $K\cap  \overline{\F}_q = F\cap  \overline{\F}_q = \F_q$.  Let $G: = \Gal(K/F)$ be the Galois group of $K/F$ and let  $\cO_F$ and $\cO_K$ be the integral closure of $A:=\F_q[t]$ in $F$ and $K$, respectively. These are Dedekind domains, consisting of all the elements in $F$ and $K$, respectively, which are integral at all valuations which do not extend $\infty$ (the normalized valuation on $\F_q(t)$ of uniformizer $1/t$). For $v\in{\rm MSpec}(\cO_F)$ we fix a decomposition group $\widetilde{G_v}\subseteq G_F$, an inertia group
$\widetilde{I_v}\subseteq \widetilde{G_v}$ and a Frobenius morphism $\widetilde{\sigma_v}\in\widetilde{G_v}$ for $v$. We let $G_v$, $I_v$ and $\sigma_v$ denote their projections via the Galois restriction map $G_F\twoheadrightarrow G$. These are the decomposition group, inertia group, and a Frobenius morphism, respectively, associated to $v$ in $K/F$.
\medskip

Next, we consider a Drinfeld module $E$ of rank $r\in\Bbb Z_{\geq 0}$, defined on $A=\F_q[t]$ with values in $\cO_F\{\tau\}$. We remind the reader that $E$ is given by an $\F_q$--algebra morphism
$$\varphi_E: \F_q[t]\to \cO_F\{\tau\}, \qquad \varphi_e(t)=t\cdot\tau^0+a_1\cdot\tau^1+\cdots +a_r\cdot\tau^r,$$
where $a_i\in\cO_F$ and $a_r\ne 0$. The Drinfeld module $E$ gives a natural functor
$$E: \left(\cO_F\{\tau\}[G]{\rm -modules}\right)\to \left(\F_q[t][G]{\rm -modules}\right), \qquad M\to E(M).$$
Of course, for any $\cO_F\{\tau\}[G]$--module $M$, the $\F_q[G]$--module structures of $M$ and $E(M)$ are identical, while the $\F_q[t]$--module structure of $E(M)$ is given by
$$t\ast m= \varphi_e(t)(m)= t\cdot m+a_1\cdot\tau^1(m)+\cdots +a_r\cdot\tau^r(m).$$

\noindent {\bf Examples.} Natural examples of the correspondence $M\to E(M)$ as above are
$$\cO_K\to E(\cO_K), \qquad \cO_K/v\to E(\cO_K/v), \qquad K_\infty:=K\otimes_{\F_q(t)}\F_q((t^{-1}))\to E(K_\infty), $$
where $v\in {\rm MSpec}(\cO_F)$ is any maximal ideal of $\cO_F$ and $K_\infty$ is the direct sum of the completions of $K$ with respect to all its valuations extending $\infty.$ Note that $ \F_q((t^{-1}))$ is the completion of $\F_q(t)$ with respect to $\infty$.

\subsection {The associated Galois representations ${\rm H}_{v_0}^1(E, G)$.}\label{S:intro-galoisrepr} For arithmetic data $(K/F, E)$ as above, any $v_0\in{\rm MSpec}(A)$ and $n\in\Bbb Z_{\geq 0}$, we let
$$E[v_0^n]:=E(\overline F)[v_0^n], \qquad T_{v_0}(E):=\varprojlim_{n}E[v_0^n]$$
be the usual $A_{v_0}$--modules of $v_0^n$--torsion points and the $v_0$--adic Tate module of $E$, endowed with the obvious $A_{v_0}$--linear, continuous $G_F$--actions. Here, $A_{v_0}$ denotes the $v_0$--adic completion of $A$ at $v_0$. Since the rank of $E$ is $r$, we have $A_{v_0}$--linear topological  isomorphisms
$$E[v_0^n]\simeq (A/v_0^n)^r, \qquad T_{v_0}(E)\simeq A_{v_0}^r.$$
Let $v\in{\rm MSpec}(\cO_F)$, such that $v\nmid v_0$. If the Drinfeld module $E$ has good reduction at $v$ (i.e. the coefficient $a_r$ of $\varphi_E(t)$ is a $v$--adic unit), then the $G_F$--representation $T_{v_0}(E)$ is unramified at $v$ and the polynomial
$$P_v(X):={\rm det}_{A_{v_0}}(X\cdot I_r-\widetilde{\sigma_v}\mid T_{v_0}(E))$$
is independent of $v_0$ and has coefficients in $A$. (See \cite{Ge91}.)

Following Goss \cite[\S 8.6]{Goss}, we let
$${\rm H}^1_{v_0}(E):=T_{v_0}(E)^\ast:={\rm Hom}_{A_{v_0}}(T_{v_0}(E), A_{v_0}),$$
endowed with the dual $G_F$--action. In analogy with abelian varieties, one should think of ${\rm H}^1_{v_0}(E)$ as the first \'etale cohomology group of $E$ with coefficients in $A_{v_0}$.
\begin{definition} We define the $G$--equivariant first \'etale cohomology groups of $E$ by
$${\rm H}^1_{v_0}(E, G):= {\rm H}^1_{v_0}(E)\otimes_{A_{v_0}}A_{v_0}[G],\qquad v_0\in{\rm MSpec}(A),$$
endowed with the diagonal $G_F$--action, where $G_F$ acts on ${\rm H}^1_{v_0}(E)$ as described above and on $A_{v_0}[G]$ via the projection
$G_F\twoheadrightarrow G$ given by Galois restriction.
\end{definition}
Note that we have an isomorphism of $A_{v_0}[G]$--modules ${\rm H}^1_{v_0}(E, G)\simeq A_{v_0}[G]^r$, for all $v_0$. The family of $A_{v_0}[G]$--linear $G_F$--representations  $\{{\rm H}^1_{v_0}(E, G)\}_{v_0}$ satisfies the properties listed in the following proposition.

\begin{proposition}\label{P:G-etale}  Let $v\in{\rm MSpec}(\cO_F)$  such that $E$ has good reduction at $v$. Let $v_0\in{\rm MSpec}(A)$, such that $v\nmid v_0$. Then the following hold.
\begin{enumerate}
\item  ${\rm H}^1_{v_0}(E, G)$ is ramified at $v$ if and only if  $v$ is ramified in $K/F$.
\item Assume that $v$ is tamely ramified in $K/F$. Then ${\rm H}^1_{v_0}(E, G)^{\widetilde I_v}$ is a finitely generated projective $A_{v_0}[G]$--module
and we have an equality
$$P_v^{\ast, G}(X):={\rm det}_{A_{v_0}[G]}(X\cdot{\rm id}-\widetilde{\sigma_v}\mid {\rm H}^1_{v_0}(E, G)^{\widetilde I_v})=\frac{X^r\cdot P_v(\sigma_v{\bf e}_v\cdot X^{-1})}{P_v(0)},$$
where ${\bf e}_v:=1/|I_v|\sum_{\sigma\in I_v}\sigma$ is the idempotent of the trivial character of $I_v$ in $A[G]$.
\item The polynomial $P_v^{\ast, G}(X)$ is independent of $v_0$ and $Nv\cdot P_v^{\ast, G}(X)\in A[G][X]$, where $Nv$ is the unique monic generator
of the ideal norm of $v$ down to $A=\Bbb F_q[t]$.
\end{enumerate}
\end{proposition}
\begin{proof} (1) follows from the fact that  ${\rm H}^1_{v_0}(E)=T_{v_0}(E)^\ast $ is unramified at $v$ (see above) and the definition
of the $G_F$--action on ${\rm H}^1_{v_0}(E, G)$.

(2) It is clear that ${\rm H}^1_{v_0}(E, G)^{\widetilde I_v}={\bf e}_v\cdot {\rm H}^1_{v_0}(E, G)$. Now, projectivity and finite generatedeness follow from
the isomorphism and equality of $A_{v_0}[G]$--modules
\begin{equation}\label {projective-cohomology}
A_{v_0}[G]^r\simeq {\rm H}^1_{v_0}(E, G)={\bf e}_v{\rm H}^1_{v_0}(E, G)\oplus(1-{\bf e}_v){\rm H}^1_{v_0}(E, G).
\end{equation}
Since the $A_{v_0}[G]$--module ${\rm H}^1_{v_0}(E, G)^{\widetilde I_v}$ is projective and finitely generated, the determinant
defining  $P_v^{\ast, G}(X)$ makes sense in $A_{v_0}[G][X]$ (see \eqref{E:det-projective} for the definition.) Now, the equality in (2) follows from \eqref{projective-cohomology}
and the remark that if $M$ is the matrix of $\widetilde{\sigma_v}$ in an $A_{v_0}$--basis $\{e_i\}_i$ of $T_{v_0}(E)$, then
the matrix of  $\widetilde{\sigma_v}$ in the ${\bf e}_vA_{v_0}[G]$--basis $\{e_i^\ast\otimes {\bf e}_v\}_i$ of
${\rm H}^1_{v_0}(E, G)^{\widetilde I_v}$ is $(\sigma_v {\bf e}_v\cdot (M^{-1})^t)$.

(3) follows from (2) and a result of Gekeler (see \cite[Thm 5.1]{Ge91}) saying that $P_v(0)$ and $Nv$ generate the same ideal in $A$.
\end{proof}

\begin{definition}\label{D:Fitt}
Let $M$ be an $A[G]$-module which is free of rank $m$ as an $\F_q[G]$-module. Then, by Proposition \ref{P:fittidealgen} in the Appendix, the Fitting ideal $\fitt^0_{A[G]}(M)$ is principal and has a unique monic generator $f_M(t)$ (viewed
as a polynomial in $t$  in $A[G]=\Bbb F_q[G][t])$ of degree equal to $m$. We define the $A[G]$--size of $M$ to be
$$|M|_{G}: = f_M(t)\in \F_q[G][t].$$
\end{definition}

The following describes a class of modules $M$ as above which will be very relevant for us.

\begin{proposition}\label{P:EF-values} For data $(K/F, E)$ as above, let $v\in{\rm MSpec}(\cO_F)$ be a prime which is tamely ramified in $K/F$.
Then the following hold.

\begin{enumerate}
\item $\cO_K/v$ and $E(\cO_K/v)$ are free $\F_q[G]$-modules of rank $[\cO_F/v:\Bbb F_q]$.
\item If $E$ has good reduction at $v$, then
$$P_{v}^{\ast, G}(1)=\frac{|E(\cO_K/v)|_G}{|\cO_K/v|_G}\in (1+t^{-1}\Bbb F_q[G][[t^{-1}]]).$$
\end{enumerate}
\end{proposition}
\begin{proof} (Sketch) Part (1) is Proposition \ref{P:EF-Carlitz}(1) of the Appendix.

We will not prove the equality in part (2) for all Drinfeld modules $E$ here, as the proof is technical and practically irrelevant for the rest of the paper. However, we give a short proof in the case where $E:=C$ is the (rank $1$) Carlitz module given by  $\varphi(t)=t+\tau$, which has good reduction at all primes of ${\rm MSpec}(\cO_F)$. In this case, it is not difficult to see that
$$P_v(X)=X-Nv.$$
According to Proposition \ref{P:G-etale}(2) above, we have
$$P_v^{\ast, G}(1)=\frac{\sigma_v{\bf e}_v-Nv}{-Nv}=\frac{Nv-\sigma_v{\bf e}_v}{Nv}.$$
Now, Proposition \ref{P:EF-Carlitz}(3) in the Appendix shows that $|C(\cO_K/v)|_G=(Nv-\sigma_v{\bf e}_v)$ and $|\cO_K/v|_G=Nv$, which concludes the proof in this case.
 Now,  for any $E$ we have
$$\frac{|E(\cO_K/v)|_G}{|\cO_K/v|_G}\in (1+t^{-1}\Bbb F_q[G][[t^{-1}]])$$
as the monic polynomials $|E(\cO_K/v)|_G$ and
$|\cO_K/v|_G$ have the same degree $[\cO_F/v:\Bbb F_q]$.
\end{proof}

\subsection{The associated $L$-functions and their special values}\label{S:intro-Lfunctions} To the data $(K/F, E)$ we associate a class of $G$--equivariant $L$--functions, generalizing
the Goss zeta function for $(F, E)$ (see \cite[\S 3]{CEGP18} for a detailed account of the relation between Goss zeta function and non-equivarient $L$-values).  In what follows, $\laurent{\F_q}{t\inv}$ is viewed as the completion of $\Fq(t)$ in the valuation at $\infty$ and $\C_\infty$ denotes the completion of an algebraic closure of $\laurent{\F_q}{t\inv}$.  For $s\in \C_\infty^\times\times \Z_p$ (Goss's space) and $f\in\F_q[t]$ monic, we let $f^s\in\C_\infty$ denote Goss's exponential (see \cite[\S 8.2]{Goss}). Under Goss's natural embedding $\Bbb Z\subseteq \C_\infty^\times\times\Z_p$ (see loc.cit.),  $f^n$ has the usual meaning for all $n\in\Z$ and $f$ as above. In particular $f^0=1$.
\begin{definition} Let $(K/F, E)$ be data as above. Its $G$-equivariant
$L$-function is given by
$$\widetilde{\Theta}^E_{K/F}: (\C_\infty^\times\times\Z_p)^+\to \C_\infty[G], \qquad \widetilde{\Theta}^E_{K/F}(s):=\widetilde{\prod_v}P_v^{\ast, G}(Nv^{-s})^{-1},$$
where the product $\widetilde{\prod}$ is taken over all $v\in{\rm MSpec}(\cO_F)$ which are {\it  tamely ramified in $K/F$} and such that {\it $E$ has good reduction at $v$}. Here $(\C_\infty^\times\times\Z_p)^+$ is a certain ``half plane'' of Goss's space, which contains $\Bbb Z_{\geq 0}$.
\end{definition}
The infinite product above converges on $(\C_\infty^\times\times\Z_p)^+$. We will not address these convergence aspects here, as we will be interested only in (a modified version of) the special value $\widetilde\Theta^E_{K/F}(0)$. According to Proposition \ref{P:EF-values}(2) above,
this special value is given by
$$\widetilde\Theta^E_{K/F}(0)=\widetilde{\prod_v}P_v^{\ast, G}(1)^{-1}=\widetilde{\prod_{v}}\frac{|\cO_K/v|_G}{|E(\cO_K/v)|_G}\in \left(1+t^{-1}\F_q[G][[t^{-1}]]\right),$$
and the convergence of the last product will emerge naturally from the proofs of our main results below. However, as Proposition \ref{P:EF-values}(2) shows, one can also consider the following convergent infinite product, taken over all
$v\in{\rm MSpec}(\cO_F)$ which are  tamely ramified in $K/F$  (regardless of the reduction type of $E$ at $v$).
$$\Theta^E_{K/F}(0):=\prod_{v \text{ tame}}\frac{|\cO_K/v|_G}{|E(\cO_K/v)|_G}\in \left(1+t^{-1}\F_q[G][[t^{-1}]]\right).$$
As it turns out, this still incomplete Euler product is not well behaved from a functional analysis point of view. As a consequence, we need to complete it by throwing in some Euler factors at those primes $v$
which are not tamely ramified in $K/F$. This is done in the following manner (see \S\ref{S:projective-modules} for details.)

\begin{definition}\label{D:taming-modules} An $\cO_F[G]\{\tau\}$--submodule $\cM$ of $\cO_K$ is called a {\it taming module for $K/F$,} or simply {\it taming module}, if it satisfies
the following properties.
\begin{enumerate}
\item $\cM$ is a projective $\cO_F[G]$--module.
\item The quotient $\cO_K/\cM$ is finite and supported only at primes $v\in{\rm MSpec}(\cO_F)$ which are not tamely ramified in $K/F$.
\end{enumerate}
\end{definition}

\begin{remark} Note that if $K/F$ is tame, then (2) above forces $\cM=\cO_K$. A well known theorem of E. Noether (see \S\ref{S:projective-modules}) shows that $\cM=\cO_K$  satisfies (1) in that case. For the existence and construction of such modules $\cM$ in general, see Proposition \ref{P:taming-module}.
\end{remark}

As shown in Proposition \ref{P:taming-module}, any taming module $\cM$ as above satisfies the following additional properties.
\begin{enumerate}
\item[(1')] $\cM/v $ is $\F_q[G]$--free of rank $[\cO_F/v:\F_q]$, for all $v\in{\rm MSpec}(\cO_F)$.
\item[(2')] $\cM/v=\cO_K/v$, for all $v$ tamely ramified in $K/F$.
\end{enumerate}
Consequently, for every taming module $\cM$ one can consider the following {\it complete} infinite Euler product, taken over all primes $v\in{\rm MSpec}(\cO_F)$.
$$\Theta^{E, \cM}_{K/F}(0):=\prod_{v}\frac{|\cM/v|_G}{|E(\cM/v)|_G}=\Theta^E_{K/F}(0)\cdot\prod_{v\text{ wild}}\frac{|\cM/v|_G}{|E(\cM/v)|_G}.$$
Note that although $\Theta^{E, \cM}_{K/F}(0)$ and $\Theta^E_{K/F}(0)$ are elements in $(1+t^{-1}\F_q[G][[t^{-1}]])$, so in general they are transcendental over $\F_q(t)[G]$, we have
$$\Theta^{E, \cM}_{K/F}(0)/\Theta^E_{K/F}(0)\in\F_q(t)[G]^\times.$$
Obviously, if $K/F$ is tame then $\Theta^{E, \cM}_{K/F}(0)=\Theta^{E}_{K/F}(0)$, as $\cM=\cO_K$ in that case.

\subsection{The associated compact $A[G]$--modules and their volumes}\label{S:intro-volumes} To the arithmetic data $(K/F, E)$, we associate a class of compact $A[G]$--modules on which we define
a multiplicative measure (volume) with values in $\F_q((t^{-1}))[G]^+$, the {\it subgroup of monic elements in $\F_q((t^{-1}))[G]^\times$,} to be defined in \S\ref{S:monic}. Recall that $A:=\F_q[t]$.

As before, $K_\infty := K \otimes_{\F_q(t)} \laurent{\F_q}{t\inv}$ is the product of the completions of $K$ at all primes above $\infty$, endowed with the usual (product) topology. It is a locally compact
$\F_q$--algebra, endowed with a natural topological $\laurent{\F_q}{t\inv}[G]$--module structure. The additive Hilbert theorem 90 shows that one has an isomorphism of (topological) $\laurent{\F_q}{t\inv}[G]$--modules
\begin{equation}\label{iso}
K_\infty\simeq \laurent{\F_q}{t\inv}[G]^n,
\end{equation}
where $n:=[F:\F_q(t)]$. Therefore $K_\infty$ is $G$-c.t. (Throughout, $G$-c.t. stands for $G$--cohomologically trivial, see \S\ref{S:appendix} for the definition.)  The ring $\cO_K$ sits naturally inside $K_\infty$ (diagonally embedded into the completions) as a discrete, cocompact $A[G]$--submodule (see \cite{T12}). Unless $K/F$ is tame, $\cO_K$ is not $G$-c.t. Further, since $K_\infty$ and its subring $\cO_K$ are naturally $\cO_F\{\tau\}[G]$--modules as well, $E(K_\infty)$ and $E(\cO_K)$ have natural $A[G]$--module structures. The first is $G$--c.t., the second is not $G$--c.t. unless $K/F$ is tame.

\begin{definition}\label{D:lattices} With notations as above, we define the following.
\begin{enumerate}
\item An $A$--lattice in $K_\infty$ is a free $A$--submodule of $K_\infty$ of rank equal to
${\rm dim}_{\F_q((t^{-1}))}K_\infty$, which spans $K_\infty$ as an
$\F_q((t^{-1}))$--vector space.
\item An $A[G]$--lattice in $K_\infty$ is an $A[G]$--submodule of $K_\infty$ which is an $A$--lattice in $K_\infty$.
\item A projective (respectively, free) $A[G]$--lattice in $K_\infty$ is an $A[G]$--lattice in $K_\infty$ which is projective (respectively, free) as an $A[G]$--module.
\end{enumerate}
\end{definition}

\begin{remark} Note that $A$--lattices in $K_\infty$ are {\it discrete and cocompact in $K_\infty$} (because $A$ is discrete and cocompact in $\F_q((t^{-1}))$).
Also, note that an $A[G]$--lattice in $K_\infty$ is projective if and only if it is $G$--c.t. (See Lemma \ref{L:gct-Dedekind}.) Further, any projective $A[G]$--lattice in $K_\infty$
is of constant local rank $n$ (as a projective $A[G]$--module), as a consequence of $\eqref{iso}$ above.
\end{remark}
\medskip
\noindent{\bf Examples.} $\cO_K$ is an $A[G]$--lattice in $K_\infty$ which is projective if and only if $K/F$ is tame. However, any taming module $\cM$ is a projective $A[G]$--lattice in $K_\infty$.
\begin{definition} We let ${\rm exp}_E$ denote the exponential of the Drinfeld module $E$. Recall that this is the unique power
series in $F_\infty[[z]]$, of the form ${\rm exp}_E(z)=z+a_1z^q+a_2z^{q^2}+\dots$,  and satisfying the functional equations
$${\rm exp}_E(aX)=\varphi_E(a)({\rm exp}_E(X)),$$
for all $a\in A$. (see \cite[Prop. 2]{T12} for existence and uniqueness.)
\end{definition}
Recall that ${\rm exp}_E$ converges everywhere on $\C_\infty$ and gives a continuous, open morphism of $A$--modules
$${\rm exp}_E: K_\infty\to E(K_\infty).$$
The uniqueness of ${\rm exp}_E$ implies that the above is in fact a morphism of $A[G]$--modules. Also, since the preimage
${\rm exp}_E^{-1}(\cO_K)$ is an $A[G]$--lattice in $K_\infty$ (see \cite[Prop. 3]{T12}), it is easy to see that the preimage ${\rm exp}_E^{-1}(\cM)$,
is also an $A[G]$--lattices in $K_\infty$, for all taming modules $\cM$. Consequently, if $\cM$ is either $\cO_K$ or a taming module and ${\widetilde{{\rm exp}_E}}: K_\infty/{\rm exp}_E^{-1}(\cM)\to E(K_\infty)/E(\cM)$ is the map induced by ${\rm exp}_E$, we have an exact sequence of compact topological $A[G]$---modules
\begin{equation}\label{E:exponential-sequence}
0\to K_\infty/{\rm exp}_E^{-1}(\cM)\overset{\widetilde{{\rm exp}_E}}\longrightarrow E(K_\infty)/E(\cM)\overset{\pi}\longrightarrow H(E/\cM)\to 0.
\end{equation}
Here $H(E/\cM)$ is defined to be the $A[G]$--module cokernel of the exponential map, i.e.
\begin{equation}\label{D:Hdef}
H(E/\cM):=\frac{E(K_\infty)}{E(\cM)+{\rm exp}_E(K_\infty)}.
\end{equation}
Note that since $E(K_\infty)/E(\cM)$ is compact and ${\rm exp}_E$ is an open map, the $A[G]$--module $H(E/\cM)$ is finite. These finite $A[G]$--modules are generalizations
of Taelman's ``class group'' $H(E/K)$ (see \cite{T12}) which is precisely $H(E/\cO_K)$, in our notation. For all $\cM$ as above, since $\cM\subseteq \cO_K$,
the exact sequences \eqref{E:exponential-sequence}
induce natural surjective $A[G]$--linear maps
\begin{equation}\label{E:class-group-surjections} H(E/\cM)\twoheadrightarrow H(E/\cO_K).\end{equation}
\begin{definition} For any module $\cM$ as above (i.e. either equal to $\cO_K$ or a taming module for $K/F$), the finite
$A[G]$-module $H(E/\cM)$ will be called the $\cM$-class group of $E$.
\end{definition}
The compact $A[G]$--modules which play an important role in what follows are $E(K_\infty)/E(\cM)$ and $K_\infty/\cM$, where $\cM$ is a taming module for $K/F$. According to \eqref{E:exponential-sequence}, these belong to the larger class $\cC$ of compact $A[G]$--modules defined below.
\begin{definition}\label{D:classC}
We let $\cC$ denote the class of compact $A[G]$-modules $M$ which are $G$--c.t. and fit in a short
exact sequence of topological $A[G]$--modules
$$0\to K_\infty/\Lambda\to M\to H\to 0,$$
where $\Lambda$ is an $A[G]$--lattice in $K_\infty$, $K_\infty/\Lambda$ is endowed with the usual (quotient) topology  and $H$ is a finite $A[G]$--module.
\end{definition}

\begin{remark}Note that $E(K_\infty)/E(\cO_K)$ belongs to the class $\cC$ if and only if $K/F$ is tame. Also, note that if $\Lambda$ is a projective $A[G]$--lattice in $K_\infty$, then $K_\infty/\Lambda$ belongs to $\cC$. \end{remark}

In \S\ref{S:lattice-index} below, we define a lattice index
$$[\Lambda_1:\Lambda_2]_G\in  \F_q((t^{-1}))[G]^+,$$
for any two projective $A[G]$--lattices $\Lambda_1$ and $\Lambda_2$  in $K_\infty$. If $G$ is trivial, this recovers Taelman's lattice index defined in \cite{T12}.
In \S\ref{S:volume-function} below, we fix an arbitrary free $A[G]$--lattice $\Lambda_0$ in $K_\infty$ and use the lattice index to define a volume function
$${\rm Vol}: \cC\to \F_q((t^{-1}))[G]^+,$$
normalized so that ${\rm Vol}(K_\infty/\Lambda_0)=1$. Here,  $\F_q((t^{-1}))[G]^+$ denotes the subgroup of monic elements in  $\F_q((t^{-1}))[G]^\times$, to be defined
in \S\ref{S:monic} below.

\subsection{The equivariant Tamagawa number formula and applications}\label{S:intro-ETNF} Our main result is the following $G$--Equivariant Tamagawa Number Formula, which generalizes Taelman's class number formula \cite{T12} to the current $G$--equivariant context. (See \S\ref{S:Main-ETNF} for the proof.)
\begin{theorem}[the ETNF for Drinfeld modules]\label{T:ETNF}  If $\cM$ is a taming module for $K/F$ and $E$ is a Drinfeld module of structural morphism
$\varphi_E:\Fq[t]\to\cO_F\{\tau\}$, then we have the following equality in $(1+t^{-1}\F_q[[t^{-1}]][G])$.
$$\Theta_{K/F}^{E, \cM}(0)=\frac{{\rm Vol}(E(K_\infty)/E(\cM))}{{\rm Vol}(K_\infty/\cM)}.$$
\end{theorem}
\begin{remark}Note that although for $M, M'\in\cC$ (e.g. $M=E(K_\infty)/E(\cM)$ and $M'=K_\infty/\cM$) the individual volumes ${\rm Vol}(M)$ and ${\rm Vol}(M')$ depend on the choice of the normalizing lattice $\Lambda_0$, the quotient ${\rm Vol}(M)/{\rm Vol}(M')$ is independent of that choice. (See \S\ref{S:volume-function} for details.)
\end{remark}
Noting that if $p\nmid |G|$ then every $A[G]$--lattice in $K_\infty$ is a projective $A[G]$--lattice (as it is $G$--c.t.), so in particular $\cO_K$ and ${\rm exp}_E^{-1}(\cO_K)$ are projective $A[G]$--lattices, we obtain the following Corollary
from the above theorem. (See \S\ref{S:Main-ETNF} for the proof.)
\begin{corollary}\label{C:G-not-p} If $p\nmid |G|$, then we have the following equality in $(1+t^{-1}\F_q[[t^{-1}]][G])$:
$$\Theta_{K/F}^E(0)=[\cO_K: {\rm exp}^{-1}_E(\cO_K)]_G\cdot |H(E/\cO_K)|_G.$$
\end{corollary}

\begin{remark} If $G$ is the trivial group (i.e. $K=F$), the above Corollary is precisely Taelman's class number formula \cite[Thm. 1]{T12}.  For a general $G$ of order coprime to $p$, the above Corollary implies the main result of Angles--Taelman \cite{AT15}. See Remark \ref{R:Taelman-Angles-Taelman} for more details.\end{remark}
\medskip

The main application of Theorem \ref{T:ETNF} above included in this paper is the Drinfeld module analogue of the classical refined Brumer--Stark Conjecture for number fields. We remind the reader that this conjecture roughly states that the special value $\Theta_{K/F,T}(0)$ of a $G$--equivariant, Euler--modified, Artin $L$--function $\Theta_{K/F,T}:\Bbb C\to\Bbb C$, associated to an abelian extension $K/F$ of number fields of Galois group $G$, belongs to the Fitting ideal $\fitt^0_{\Z[G]}({\rm Cl}_{K,T}^{\vee})$ of the Pontrjagin dual of a certain ray--class group ${\rm Cl}_{K,T}$ of the top field $K$. (See \cite[\S 6.1]{GP15} for a precise statement and conditional proof.) This conjecture has tremendously far reaching applications to the arithmetic of number fields. (See \cite{BG19} for details.) The Drinfeld module analogue of this conjecture is the following. (See \S\ref{S:Main-Brumer-Stark} for the proof.)

\begin{theorem} [refined Brumer--Stark for Drinfeld modules]\label{T:BrSt}  If $\cM$ is a taming module for $K/F$, $E$ is a Drinfeld module of structural morphism
$\varphi_E:\Fq[t]\to\cO_F\{\tau\}$,  and $\Lambda'$ is a $E(K_\infty)/E(\cM)$--admissible $A[G]$--lattice in $K_\infty$, then we have
$$\frac{1}{[\cM:\Lambda']_G}\cdot \Theta_{K/F}^{E, \cM}(0)\in{\rm Fitt}^0_{A[G]}H(E/\cM).$$
\end{theorem}
\begin{remark}
For every taming module $\cM$, we define in \S\ref{S:volume-function} a class of projective $A[G]$--lattices $\Lambda'$ which we call {\it $E(K_\infty)/E(\cM)$--admissible} and which are instrumental in defining the volume ${\rm Vol}(E(K_\infty)/E(\cM)).$
\end{remark}

The above Theorem has the following two consequences regarding the $A[G]$--module structure of Taelman's ideal--class group $H(E/\cO_K).$ (See \S\ref{S:Main-Brumer-Stark} for proofs and additional remarks.)
\begin{corollary}
With notations as in Theorem \ref{T:BrSt}, we have
$$\frac{1}{[\cM:\Lambda']_G}\cdot \Theta_{K/F}^{E, \cM}(0)\in{\rm Fitt}^0_{A[G]}H(E/\cO_K).$$
\end{corollary}
\noindent In the case $p\nmid |G|$, the lattice ${\rm exp}^{-1}_E(\cO_K)$ is $K_\infty/\cO_K$--admissible. Consequently, we obtain a description of the full Fitting ideal of $H(E/\cO_K)$ in this case.
\begin{corollary}\label{C:fullFitt}
If $p\nmid |G|$, then we have an equality of principal $A[G]$--ideals
$$\frac{1}{[\cO_K: {\rm exp}^{-1}_E(\cO_K)]_G}\Theta_{K/F}^E(0)\cdot A[G]={\rm Fitt}^0_{A[G]} H(E/\cO_K).$$
\end{corollary}

\begin{remark}
In the number field vs. Drinfeld module analogy, the $T$--modified $L$--value $\Theta_{K/F, T}(0)$ corresponds to the $\cM$--modified $L$--value $\left(\frac{1}{[\cM:\Lambda']_G}\cdot \Theta_{K/F}^{E, \cM}(0)\right)$. At the same time, the natural class--group surjection ${\rm Cl}_{K,T}\twoheadrightarrow{\rm Cl}_K$ corresponds to
the equally natural surjection $H(E/\cM)\twoheadrightarrow H(E/\cO_K)$. See more on this analogy in \S\ref{S:Main-Brumer-Stark}.
\end{remark}

\subsection{A brief word on proof strategy and techniques} Once we construct and study the various invariants associated to the data $(K/F, E)$ and briefly described in Sections 1.2--1.4 of this introduction, the proofs of the main results stated above rely on $G$--equivariant versions of Taelman's techniques (\cite{T12}),
which we develop in this paper. In particular, we prove a $G$--equivariant version of Taelman's trace formula (see \S\ref{S:Trace})), which plays a crucial role in obtaining Theorem \ref{T:ETNF}. The main obstacle for passing from a non-equivariant to a $G$--equivariant setting is, as expected, lack of cohomological triviality (or lack of finite projective dimension) of the various $A[G]$--modules at play. Of course, this obstacle would not be present had we assumed that $p\nmid |G|$, as in \cite{AT15} and \cite{Fang15}, for example.

\section{Nuclear operators, the $G$-equivariant theory}\label{S:Nuclear}

\subsection{Generalities}

Let $R:=\Fq[G]$ and let $V$ be a topological $R$--module, which is $R$--projective or, equivalently, $G$--c.t. (See Corollary \ref{C:gct-arbitraryG}(1) for the equivalence.) In this section, we develop the theory of nuclear operators and determinants a la Taelman (see \cite[\S 2]{T12}) for $V$ as an $R$-module as opposed to $\F_q$-vector space.  The main difference between the $R$-linear and $\Fq$--linear settings is that in the $R$--linear setting one can only take determinants of endomorphisms of finitely generated, projective $R$--modules (as opposed to any finite dimensional $\Fq$--vector spaces), in the sense of \eqref{E:det-projective} in the Appendix. In what follows, ``endomorphism of $V$'' means a continuous $R$--module endomorphism of $V$.

\begin{definition}\label{D:Umdef}
Let $\mathcal{U} = \{U_i\}_{i\geq M}$ be a sequence of open $R$-submodules of $V$ with the following properties:
\begin{enumerate}
\item[(1)] Each $U_i$ is $G$--c.t.;
\item[(2)] $U_{i+1}\subseteq U_i$, for all $i\geq M$;
\item[(3)] $\mathcal U$ forms a basis of open neighborhoods of $0$ in $V$.
\end{enumerate}
\end{definition}
Assuming that $\mathcal{U}$ exists, we fix it and define everything that follows for the pair $(V,\mathcal{U})$. Independence on $\cU$ in the definitions and results below will be addressed in \S\ref{S:independence-on-U}.

\begin{definition} Let $\varphi$ be an endomorphism of $V$. We say that $\varphi$ is locally contracting if there exists an $I\in\Z_{\geq M}$, such that $\varphi(U_i)\subseteq U_{i+1}$, for all $i\geq I$. A neighborhood $U := U_I$ of $0$ with this property is called a nucleus for $\varphi$. \end{definition}

\begin{remark}If $V$ is a finitely generated $R$-module, then we always take $U_i = \{0\}$, for all $i\geq 1$. Obviously, every endomorphism of $V$ is locally contracting in this case.
\end{remark}
The following are clear.

\begin{proposition} Any finite collection of locally contracting endomorphisms of $V$ has a common nucleus.\end{proposition}

\begin{proposition} \label{SumAndComposition}If $\varphi$ and $\psi$ are locally contracting endomorphisms of $V$, then so are the sum $\varphi + \psi$ and the composition $\varphi\psi$.\end{proposition}

Following Taelman \cite{T12}, we let $R[[Z]]$ be the ring of power series in variable $Z$ with coefficients in $R$ and consider the $R[[Z]]$--modules
$$V[[Z]]/Z^N := V\otimes_R R[[Z]]/Z^N,  \text{ and } V[[Z]] := \varprojlim_{N\geq 1} V[[Z]]/Z^N.$$
We endow $V[[Z]]/Z^N$ with the product topology of $N$ copies of $V$ and $V[[Z]]$ with the inverse limit topology. These are topological $R[[Z]]$--modules, where $R[[Z]]$ is endowed with its $Z$--adic topology. It is easily seen that any continuous $R[[Z]]$-linear endomorphism $\Phi$ of $V[[Z]]$ (respectively $R[[Z]]/Z^N$--linear endomorphism of $V[[Z]]/Z^N$) is of the form
\begin{equation}\label{E:Phi-sum}\Phi = \sum_{n = 0}^\infty \varphi_nZ^n \text{ (respectively $\Phi = \sum_{n = 0}^{N-1}\varphi_nZ^n$)},\end{equation}
where the $\varphi_n$'s are uniquely determined endomorphisms of $V$.

\begin{remark}If $V$ is a finitely generated, projective $R$-module, then $V[[Z]]/Z^N$ and $V[[Z]]$ are finitely generated, projective $R[[Z]]/Z^N$-- and $R[[Z]]$--modules, respectively. (Note that for such $V$'s we have an isomorphism $V[[Z]]\simeq V\otimes_RR[[Z]]$ of $R[[Z]]$--modules.) Therefore, we may take determinants of endomorphisms $\Phi$ of $V[[Z]]/Z^N$ and $V[[Z]]$ in the classical sense, as defined in \eqref{E:det-projective} of the Appendix. For notational convenience, in this case we let
$$\det\nolimits_{R[[Z]]}(\Phi\rvert V):= \det\nolimits_{R[[Z]]}(\Phi\rvert V[[Z]]), \quad \det\nolimits_{R[[Z]]/Z^N}(\Phi\rvert V):=\det\nolimits_{R[[Z]]/Z^N}(\Phi\rvert V[[Z]]/Z^N).$$
\end{remark}
\medskip

For the rest of this section, we assume that $V$ is compact, but not necessarily finitely generated over $R$. Now, we describe how to take determinants of certain types of endomorphisms of $V[[Z]]/Z^N$ and $V[[Z]]$ in this more general setting. Note that for all $j\geq i\geq M$, the $R$--modules $V/U_i$ and $U_i/U_j$ are finite, therefore finitely generated and projective. (Since $V$ and the $U_i$'s are all $G$--c.t., by assumption, and therefore
$V/U_i$ and $U_i/U_j$ are all $G$--c.t.)

\begin{definition}\label{D:Nuclear} We say that a continuous $R[[Z]]$-linear endomorphism $\Phi$ of $V[[Z]]$ (respectively $V[[Z]]/Z^N$) is nuclear, if for all $n\geq 0$ (respectively all $n$, with $N>n\geq 0$), the endomorphisms $\varphi_n$ of $V$ defined in \eqref{E:Phi-sum} are locally contracting.\end{definition}

\begin{proposition} \label{NuclearNucleus}Let $\Phi:V[[Z]]/Z^N\rightarrow V[[Z]]/Z^N$ be a nuclear endomorphism. Let $U = U_J$ and $W = U_I$ be common nuclei for all the corresponding $\varphi_n$'s. Then
$$\det\nolimits_{R[[Z]]/Z^N}(1 + \Phi\rvert V/U) = \det\nolimits_{R[[Z]]/Z^N}(1 + \Phi\rvert V/W).$$
\end{proposition}

\begin{proof} Say $I\leq J$, so $U\subseteq W$. Then, we have the descending sequence
$$W = U_I\supseteq U_{I+1}\supseteq U_{I+2}\supseteq\cdots\supseteq U_{J-1}\supseteq U_J = U,$$
such that $\varphi_n(U_i)\subseteq U_{i + 1}$ for all $n$ and $i$, with $0\leq n<N$ and $I\leq i\leq J-1$. Then $(1 + \Phi)$ induces the identity map on the quotients $U_i/U_{i + 1}$, so we have
$$\begin{array}{ll} \det\nolimits_{R[[Z]]/Z^N}(1 + \Phi\rvert V/U) &\ds = \det\nolimits_{R[[Z]]/Z^N}(1 + \Phi\rvert V/W)\prod_{i = I}^{J - 1}\det\nolimits_{R[[Z]]/Z^N}(1 + \Phi\rvert U_i/U_{i + 1})\\
													    & = \det\nolimits_{R[[Z]]/Z^N}(1 + \Phi\rvert V/W).\end{array}$$\end{proof}

\begin{definition} Let $\Phi$ be a nuclear endomorphism of $V[[Z]]/Z^N$. Then we define
$$\det\nolimits_{R[[Z]]/Z^N}(1 + \Phi\rvert V) := \det\nolimits_{R[[Z]]/Z^N}(1 + \Phi\rvert V/U)$$
where $U$ is any common nucleus for the corresponding $\varphi_n$'s. If $\Phi$ is a nuclear endomorphism of $V[[Z]]$, then we define the determinant
of $(1+\Phi)$ in $\ds R[[Z]] = \varprojlim_N R[[Z]]/Z^N$ by
$$\det\nolimits_{R[[Z]]}(1 + \Phi\rvert V):=\varprojlim_N\det\nolimits_{R[[Z]]/Z^N}(1 + \Phi\rvert V).$$
(The reader has to check that the projective limit above makes sense.) \end{definition}

\begin{remark}\label{R:Nuclear-Fitting} Assume that $M$ is a finite $R[t]$--module (i.e. an $A[G]$--module, where $A=\Fq[t]$) which is $R$--free of rank $n$. Then we can view $\Phi:=-t\cdot T^{-1}$ as a nuclear endomorphism of $M[[T^{-1}]]$. Then $\det_{R[[T^{-1}]]}(1-t\cdot T^{-1}\mid M)$ as defined above is the usual determinant of $(1+\Phi)$ viewed as an endomorphism of the free $R[[T^{-1}]]$--module $M\otimes_RR[[T^{-1}]]$ of rank $n$.
Proposition \ref{P:fittidealgen}(1) then gives the following equality in $R[t]$:
$$|M|_G= t^n\cdot\det\nolimits_{R[[T^{-1}]]}(1-t\cdot T^{-1}\mid M)\rvert_{T=t}.$$
\end{remark}

\begin{proposition}\label{MultDet} Let $\Phi$ and $\Psi$ be nuclear endomorphisms of $V[[Z]]$. Then  the endomorphism $(1 + \Phi)(1 + \Psi) - 1$ is nuclear, and
$$\det\nolimits_{R[[Z]]}((1 + \Phi)(1 + \Psi)\rvert V) = \det\nolimits_{R[[Z]]}(1 + \Phi\rvert V)\det\nolimits_{R[[Z]]}(1 + \Psi\rvert V).$$\end{proposition}

\begin{proof} This follows from Proposition \ref{SumAndComposition} and the multiplicativity of finite determinants.\end{proof}

\begin{proposition} \label{MultSES} Let $V'\subseteq V$ be a closed $R$-submodule of $V$ which is $G$--c.t. and let $V'' := V/V'$. Let $\mathcal{U}' = \{U_i'\}_i$ where $U_i' = U_i\cap V'$, and $\mathcal{U}'' = \{U_i''\}_i$ where $U_i''$ is the image of $U_i$ in $V''$. Assume that all the $U_i'$' and $U_i''$ are $G$-c.t. Let $\Phi = \sum \varphi_nZ^n:V[[Z]]\rightarrow V[[Z]]$ be a nuclear endomorphism, such that $\varphi_n(V')\subseteq V'$, for all $n$. Then the endomorphisms induced by $\Phi$ on $(V', \mathcal{U}')$ and $(V'',\mathcal{U}'')$ are nuclear and
$$\det\nolimits_{R[[Z]]}(1 + \Phi\rvert V) = \det\nolimits_{R[[Z]]}(1 + \Phi\rvert V')\det\nolimits_{R[[Z]]}(1 + \Phi\rvert V'').$$
\end{proposition}

\begin{proof}Clear from the behaviour of finite determinants in short exact sequences.\end{proof}

\subsection{Independence of $\cU$}\label{S:independence-on-U} Assume that $V$ is a compact, $G$--c.t. $R$--module and that $\cU=\{U_i\}_i$ and $\cU'=\{U'_i\}_i$ are two bases of open neighborhoods of $0$ in $V$, satisfying the properties in Definition \ref{D:Umdef}. Let $\varphi\in{\rm End}_R(V)$ and $\Phi=\sum_n\varphi_nZ^n\in{\rm End}_{R[[Z]]}(V[[Z]])$ be such that $\varphi$ is locally contracting and $\Phi$ is nuclear with respect to both $\cU$ and $\cU'$.
\begin{definition}\label{D: phi-domination} We say that $\cU$ $\varphi$--dominates $\cU'$, and write $\cU\succeq_{\varphi}\cU'$,  if there exists an $M\in\Z_{\geq 0}$ such that
for all $i\geq M$ there exists $j\geq M$ satisfying
$$U_{i}\supseteq U'_{j} \text{ and } \varphi(U_{i})\subseteq U'_{j}.$$
We say that $\cU$ $\Phi$--dominates $\cU'$, and write $\cU\succeq_{\Phi}\cU'$, if $\cU\succeq_{\varphi_n}\cU'$, for all $n\geq 0$.
\end{definition}
\begin{lemma}\label{L:independence-on-U} Assume that $V$, $\Phi$, $\cU$ and $\cU'$ are as above, and $\cU\succeq_{\Phi}\cU'$. Then
$${\rm det}_{R[[Z]]}(1+\Phi|V)={\rm det}{'}_{R[[Z]]}(1+\Phi|V),$$
where the nuclear determinants ${\rm det}$ and ${\rm det}'$ are computed with respect to $\cU$ and $\cU'$, respectively.
\end{lemma}
\begin{proof} Let $N\in\Bbb Z_{\geq 1}$. It is easy to see that we can take $i$ and $j$ sufficiently large, such that
$$U_i\supseteq U_j', \qquad \varphi_n(U_i)\subseteq U'_j, \text{ for all }n<N$$
and such that $U_i$ and $U_j'$ are common nuclei for $\varphi_0, \dots, \varphi_{N-1}$. Consider the exact sequence of finite, $G$--c.t. $R$--modules
$$0\to U_i/U'_j\to V/U'_j\to V/U_i\to 0,$$
and note that $\varphi_n$ gives the $0$--map when restricted to $U_i/U'_j$, for all $n<N$. Consequently, the exact sequence above gives an equality of (regular) determinants
$${\rm det}_{R[[Z]]/Z^N}(1+\Phi|V/U_i)={\rm det}_{R[[Z]]/Z^N}(1+\Phi|V/U'_j).$$
This yields the desired equality of nuclear determinants by taking a limit when $N\to\infty$.
\end{proof}

\subsection{The relevant compact $R$--modules $V$.} Now, we construct two examples of compact, projective $R$--modules $V$ and corresponding bases $\mathcal U$ of open, $G$--c.t. submodules as above. For that purpose we fix what we call a taming pair $(\mathcal W, \mathcal W^\infty)$ for $K/F$, consisting of a taming module $\mathcal W$  and an $\infty$--taming module $\mathcal W^\infty$ for $K/F$. (See Definition \ref{D:infty-taming} and Proposition \ref{P:taming-module} for the properties and existence of $\mathcal W$ and $\mathcal W^\infty$.)

For a prime $v$ in $F$, we let $K_v:=\prod_{w|v}K_w$ be the product of the $w$--adic completions of $K$, for all primes $w$ in $K$ sitting above $v$, endowed with the product of the $w$--adic topologies. As usual, $F_v$, $\cO_v$, and $\frak m_v$ denote the $v$--adic completion of $F$, its ring of integers, and the maximal ideal of that ring, respectively. We denote by $S_\infty$ the set of infinite primes in $F$ and let $K_\infty=\prod_{v\in S_\infty}K_v$. For a prime $v$  in $F$, we let $\mathcal W_v^\infty$ and $\mathcal W_v$ denote the $v$--adic completion of $\mathcal W^\infty$, if $v\in S_\infty$, and the $v$--adic completion of $\mathcal W$, if $v\not\in S_\infty$, respectively. These are $\cO_v[G]$--submodules of $K_v$, for all $v$. Recall that Corollary \ref{C:taming-basis} shows that if $v\in S_\infty$ and $v\not\in S_\infty$, respectively, then
\begin{equation}\label{E:nuclei}
U_{i, v}:=\{t^{-i}\mathcal W_v^\infty\}_{i\geq 0}, \qquad\qquad U_{i,v}:=\{\frak m_v^i\mathcal W_v\}_{i\geq 0}\end{equation}
give bases of open neighborhoods of $0$ in $K_v$, consisting of free $\cO_v[G]$--submodules of rank 1, therefore projective $R$--submodules (as they are
$G$--c.t.) of $K_v$.

\subsubsection{The class $\mathcal C$.}\label{S:nuclear-classC} Let $V$ be an element in the class $\mathcal{C}$ of compact $A[G]$--modules given in Definition \ref{D:classC}, and let
\[0\longrightarrow K_\infty/\Lambda \overset{\iota}\longrightarrow V\longrightarrow H\longrightarrow 0\]
be a structural exact sequence for $V$ as in loc.cit. In particular, $V$ is a compact $R$--module. To construct a sequence $\mathcal U$ of open $R$--submodules for $V$ as in Definition \ref{D:Umdef}, we give a basis of open $R$-submodules of $K_\infty$ which are $G$-c.t. This will induce an appropriate basis of open submodules of $V$ as described below.

 For all $i\geq 0$, we  let
\begin{equation}\label{E:ui-infty}U_{i,\infty} := \prod_{v\in S_\infty} U_{i,v}\subseteq K_\infty=\prod_{v\in S_\infty}K_v.\end{equation}
According to Corollary \ref{C:taming-basis} the $\{U_{i, \infty}\}_{i\geq 0}$ are compact, open, $G$-c.t. $R$--submodules of $K_\infty$ which form a basis of open neighborhoods of $0$ in $K_\infty$.

Recalling that $\Lambda$ is discrete in $K_\infty$, let $\ell\geq 1$ be such that $U_{\ell,\infty}\cap \Lambda = \{0\}$. For $i\geq \ell$, we identify $U_{i,\infty}$ with its image in $V$ via $\iota$, and define $\mathcal{U} = \{U_{i,\infty}\}_{i\geq \ell}$ as the appropriate basis of open neighborhoods of $0$ in $V$. Now, we can define nuclear endomorphisms and take nuclear determinants for the pair $(V, \mathcal{U})$.

\subsubsection{$V$'s arising from general taming modules for $K/F$.}\label{S:nuclear-taming} Let $\mathcal{M}$ be a taming module for $K/F$ as in Definition \ref{D:taming-modules}, and let $S$ be a finite set of primes of $F$ containing $S_\infty$. Let $K_S := \prod_{v\in S}K_v$, endowed with the sup norm. Let
$$\OX_{F,S} = \{\alpha\in F: v(\alpha) \geq 0, \text{ for } v\not\in S\}$$
 be the ring of $S$-integers in $F$, and let $\mathcal{M}_S = \mathcal{M}\otimes_{\OX_F}\OX_{F,S}$. The module $\mathcal{M}_S$ is discrete and cocompact in $K_S$, because $\cO_{K,S}:=\cO_K\otimes_{\cO_F}\cO_{F,S}$ is. In this case, we let
$$V:=K_S/\cM_S.$$
Now, since $K_S\simeq \prod_{v}(K_w\otimes_{\Fq[G_v]}\F_q[G])$ as $R$--modules (where for each $v$, $w$ is a place of $K$ above $v$), $K_S$ is $G$--c.t., as a consequence of the normal basis theorem and Shapiro's Lemma. Since $\cM$ is $\cO_F[G]$--projective (by definition), then $\cM_S$ is $\cO_{F, S}[G]$--projective. Therefore $\cM_S$ is $G$--c.t. Consequently $V$ is $G$--c.t. Therefore, $V$ is a topological, compact, projective $R$--module.
We give a basis of $G$-c.t., open $R$-submodules of $K_S$, which induces a basis of open $R$-submodules of $V$, as described below.

For all $i\geq 0$, we let
$$U_{i, S}:=\prod\limits_{v\in S}U_{i,v} \subseteq \prod\limits_{v\in S}K_v=K_S.$$
As above, Corollary \ref{C:taming-basis} shows that $(U_{i, S})_{i\geq 0}$ forms a basis of $R$--projective, open submodules of $K_S$. Now, since $\cM_S$ sits discretely in $K_S$, we can pick an $\ell\in\Bbb Z_{\geq 1}$, such that $U_{\ell, S}\cap \cM_S=\{0\}$. Identify $U_{i,S}$ with their images in $V=K_S/\cM_S$, for all $i\geq \ell$ and define $\cU:=(U_{i, S})_{i\geq \ell}$  as the appropriate basis of open neighborhoods of $0$ in $V$. Now, we can define nuclear endomorphisms and determinants for the pair
$(V, \cU)$.

\begin{lemma} \label{SimpleLocallyContracting} Let $\mathcal{M}$ be a taming module for $K/F$, and let $S$ be a finite set of primes of $F$ containing $S_\infty$. Let $\varphi = \alpha \tau^n$, for some $\alpha \in\OX_{F,S}$ and $n\geq 1$. Then $\varphi$ is a locally contracting endomorphism of $K_S/\mathcal{M}_S$.
\end{lemma}

\begin{proof} Clearly, $\varphi$ is an endomorphism of $K_S/\mathcal{M}_S$ since $\mathcal{M}_S$ is an $\cO_{F,S}[G]\{\tau\}$--module. Let $m\in\Z_{\geq 1}$, such that $m\geq (1 - v(\alpha))$, for all $v\in S$. Recalling that, by definition, $\tau(\mathcal W_v)\subseteq \mathcal W_v$ and $\tau(\mathcal W^\infty_v)\subseteq \mathcal W^\infty_v$, we obviously have $\varphi(U_{i, S})\subseteq U_{i + 1, S}$, for all $i\geq m$.\end{proof}

\begin{corollary}\label{C:tauPolyLocallyContracting} For any $\cM$ and $S$ as in the lemma, any $\varphi\in\OX_{F,S}\{\tau\}\tau$ is a locally contracting endomorphism of $K_S/\mathcal{M}_S$. Consequently, any
$\Phi\in\OX_{F,S}\{\tau\}\tau[[Z]]$ is a nuclear endomorphism of $K_S/\mathcal{M}_S[[Z]]$.  \end{corollary}
\begin{proof} Combine the Lemma above with Proposition \ref{SumAndComposition}.
\end{proof}

\begin{proposition} \label{CommDet} Let $\mathcal{M}$ be a taming module for $K/F$, and let $S$ be a finite set of primes of $F$ containing $S_\infty$. Let $\alpha, \beta \in\OX_{F,S}$ and let $\varphi = \beta\tau^n$ for $n\geq 1$. Then for any $m\in\Z_{\geq 1}$,
$$\det\nolimits_{R[[Z]]}(1 + \alpha\varphi Z^m\rvert K_S/\mathcal{M}_S) = \det\nolimits_{R[[Z]]}(1 + \varphi\alpha Z^m\rvert K_S/\mathcal{M}_S).$$\end{proposition}

\begin{proof} We may assume that $\alpha, \beta\ne 0$. Define $\varphi_\alpha: {K_S}/{\mathcal{M}_S}\rightarrow {K_S}/{\mathcal{M}_S}$ by $\varphi_\alpha(x) = \alpha x$. Let $a\in\Z_{\geq 1}$ be such that $U_{a,S}\cap \mathcal{M}_S = \{0\}$, $U_{a,S}$ is a nucleus for $\varphi, \varphi_\alpha\varphi$, and $\varphi\varphi_\alpha$, and
$$b:=\min\{a + v(\beta) : v\in S\} > \max\left(\{0\}\cup \{-v(\alpha) : v\in S\}\right).$$
We have a commutative diagram of finite $R$--module morphisms
$$\begin{tikzpicture}[description/.style={fill=white,inner sep=2pt}]
\matrix (m) [matrix of math nodes, row sep=3em,
column sep=2.5em, text height=1.5ex, text depth=0.25ex]
{\frac{K_S/\mathcal{M}_S}{\varphi_\alpha^{-1}(U_{a,S})} & \frac{K_S/\mathcal{M}_S}{U_{a,S}}\\ \frac{K_S/\mathcal{M}_S}{\varphi_\alpha^{-1}(U_{a,S})} & \frac{K_S/\mathcal{M}_S}{U_{a,S}}\\};
\path[->, font=\scriptsize]
(m-1-1) edge node[auto] {$\varphi_\alpha$} (m-1-2)
(m-1-1) edge node[auto] {$\varphi\varphi_\alpha$} (m-2-1)
(m-2-1) edge node[auto] {$\varphi_\alpha$} (m-2-2)
(m-1-2) edge node[auto] {$\varphi_\alpha\varphi$} (m-2-2);
\end{tikzpicture}$$
whose horizontal arrows are isomorphisms (as $\alpha$ is invertible in $K_S$.)
For $a$ as above, $\varphi_\alpha^{-1}(U_{a,S})\simeq U_{a, S}$ as $R$--modules, so $\varphi_\alpha^{-1}(U_{a,S})$ is $G$-c.t. Therefore $\det\nolimits_{R[[Z]]}(1 + \varphi\varphi_\alpha Z^m\rvert \frac{K_S/\mathcal{M}_S}{\varphi_\alpha^{-1}(U_{a,S})})$ is defined. Consequently, from the above diagram, we obtain
\begin{equation}\label{E:equality1}\det\nolimits_{R[[Z]]}\left (1 + \varphi\varphi_\alpha Z^m\bigg\vert \frac{K_S/\mathcal{M}_S}{\varphi_\alpha^{-1}(U_{a,S})}\right ) = \det\nolimits_{R[[Z]]}\left (1 + \varphi_\alpha\varphi Z^m\bigg\vert\frac{K_S/\mathcal{M}_S}{U_{a,S}}\right ).\end{equation}
However, since $U_{a, S}$ is a nucleus for $\varphi_\alpha\varphi$, by definition we have
\begin{equation}\label{E:equality2}\det\nolimits_{R[[Z]]}(1 + \varphi_\alpha\varphi Z^m\rvert \frac{K_S}{\mathcal{M}_S}) = \det\nolimits_{R[[Z]]}(1 + \varphi_\alpha\varphi Z^m\rvert\frac{K_S/\mathcal{M}_S}{U_{a,S}}).\end{equation}
Now, from the definition of $b$ it is easy to see that
\begin{equation}\label{Inclusions}\varphi\varphi_\alpha(\varphi_\alpha^{-1}(U_{a, S}))\subseteq U_{a + b, S}\subseteq \varphi_\alpha^{-1}(U_{a,S}).\end{equation}
Consider the following short exact sequence of finite, projective $R$--modules.
\begin{equation}0\longrightarrow \frac{\varphi_\alpha^{-1}(U_{a,S})}{U_{a + b, S}} \longrightarrow \frac{K_S/\mathcal{M}_S}{U_{a + b, S}}\longrightarrow \frac{K_S/\mathcal{M}_S}{\varphi_\alpha^{-1}(U_{a,S})}\longrightarrow 0.\end{equation}
By (\ref{Inclusions}), we have $\det\nolimits_{R[[Z]]}(1 + \varphi\varphi_\alpha\rvert\frac{\varphi_\alpha^{-1}(U_{a,S})}{U_{a + b, S}}) = 1$. Consequently, if we combine the fact that $U_{a+b, S}$ is a nucleus for $\varphi\varphi_\alpha$ (because $U_{a, S}$ is and $b>0$) with the short exact sequence above and with \eqref{E:equality1} and \eqref{E:equality2}, we to obtain
\begin{align*}
\det\nolimits_{R[[Z]]}(1 + \varphi\varphi_\alpha Z^m\rvert K_S/\mathcal{M}_S) & = \det\nolimits_{R[[Z]]}\left (1 + \varphi\varphi_\alpha Z^m\bigg\vert\frac{K_S/\mathcal{M}_S}{U_{a + b, S}}\right )\\
 						& = \det\nolimits_{R[[Z]]}\left (1 + \varphi\varphi_\alpha Z^m\bigg\vert \frac{K_S/\mathcal{M}_S}{\varphi_\alpha^{-1}(U_{a,S})}\right )\\
&=\det\nolimits_{R[[Z]]}(1 + \varphi_\alpha\varphi Z^m\rvert K_S/\mathcal{M}_S).
\end{align*}
\end{proof}
The following Lemma addresses independence on the chosen taming pair $(\mathcal W, \mathcal W^\infty)$.
\begin{lemma}\label{L:independence-on-W} Assume that $\cM$ and $S$ are as in the last proposition and let $\Phi\in\cO_{F,S}\{\tau\}[[Z]]\tau$, viewed as
an $R[[Z]]$--endomorphism of $K_S/\cM_S[[Z]]$. Then the nuclear determinant
$${\rm det}_{R[[Z]]}(1+\Phi|K_S/\cM_S)$$
is independent of the taming pair $(\mathcal W, \mathcal W^\infty)$ for $K/F$.
\end{lemma}
\begin{proof}
Let $(\mathcal W, \mathcal W^\infty)$ and $(\mathcal W', \mathcal W'^\infty)$ be two such taming pairs for $K/F$. Let $\cU$ and $\cU'$ be the bases of open neighborhoods of $0$ in $K_S/\cM_S$ constructed
as above out of these pairs, respectively. Let $\varphi\in \cO_{F,S}\{\tau\}\tau$. Then, we claim that
\begin{equation}\label{E:dominance}\cU\succeq_{\varphi}\cU'.\end{equation}
Indeed, since the completions $\mathcal W_v, \mathcal W'_v$ (for $v\in S\setminus S_\infty$) and $\mathcal W^\infty_v, \mathcal W'^\infty_v$ (for $v\in S_\infty$) are open in the field
completions $K_v$ and the set $S$ is finite, there exist $a, b\in\Z_{\geq 0}$ such that
$$U'_{i+a, S}\subseteq U_{i,S}, \quad U_{i+b, S}\subseteq U'_{i, S}, \quad \text{ for all }i\gg 0.$$
Since $\varphi\in \cO_{F,S}\{\tau\}\tau$, it is easy to see that there exists $\alpha\in\Z_{\geq 0}$ (depending on the coefficients in the $\tau$--expansion of $\varphi$) such that $\varphi(U_{i,S})\subseteq U_{iq-\alpha, S}$, for all $i\gg 0$. This shows that
$$U'_{i+a, S}\subseteq U_{i,S} \text { and } \varphi(U_{i,S})\subseteq U'_{i+a, S}, \text{ for all }i\gg a+b+\alpha.$$
Therefore \ref{E:dominance} holds. Now, the desired result follows by applying Lemma \ref{L:independence-on-U}.
\end{proof}

\section{The $G$-equivariant trace formula and consequences}\label{S:Trace}

In this section we prove a trace formula for $\Fq[G]$--linear nuclear operators on $K_S/\cM_S$ by using the line of reasoning in \cite[\S 3]{T12}, adapted to our $G$-equivariant setting. As a consequence, we interpret the special values $\Theta_{K/F}^{E, \cM}(0)$ of the $G$--equivariant $L$--functions defined in the introduction as determinants of such a nuclear operators. The notations are as above.

\begin{lemma}\label{LocalizationLemma} Let $\mathcal{M}$ be a taming module for $K/F$. Let $S$ be a finite set of primes of $F$ containing $S_\infty$, let $v\in {\rm MSpec}(\cO_F)\setminus S$, and let $S' := S\cup\{v\}$. Then, for any operator $\Phi\in\OX_{F,S}\{\tau\}[[Z]]\tau Z$, we have
$$\det\nolimits_{R[[Z]]}(1 + \Phi\rvert \mathcal{M}/v\mathcal{M}) = \frac{\det\nolimits_{R[[Z]]}(1 + \Phi\rvert K_{S'}/\mathcal{M}_{S'})}{\det\nolimits_{R[[Z]]}(1 + \Phi\rvert K_S/\mathcal{M}_S)}.
$$ \end{lemma}

\begin{proof}
As in the proof of Lemma 1 in \cite{T12}, we have a sequence of compact, $G$--c.t. $R$--modules
$$0\longrightarrow \cM_v\overset\psi\longrightarrow\frac{K_{S'}}{\cM_{S'}}\overset\eta\longrightarrow\frac{K_S}{\cM_S}\longrightarrow 0.$$
Above, we view $K_{S'} = K_S\times K_v$. In this representation, $\psi(\alpha) = \widehat{(0,\alpha)}$, for all $\alpha\in\cM_v$. Also, for $\alpha\in K_S$ and $\beta\in K_v$, we define $\eta(\widehat{(\alpha,\beta)}) = \widehat{\alpha - \alpha'}$, where $\alpha'\in \mathcal{M}_{S'}$ is chosen such that $\beta = \alpha' + \beta'$, with $\beta'\in \mathcal{M}_v$. Such an $\alpha'$ exists since $K_v = \cO_{K,S'}+\cO_{K_v}=\cM_{S'} + \cM_v$, as one can check by applying a strong approximation theorem. It is easily seen that $\psi$ and $\eta$ are well defined and that the sequence above is exact.

Now, since we can compute the nuclear determinants in question with respect to bases of open neighborhoods of $0$ constructed out of any taming pair (see Lemma \ref{L:independence-on-W}), we choose to work
with the taming pair $(\cM, \mathcal W^\infty)$, where $\mathcal W^\infty$ is an arbitrary $\infty$--taming module for $K/F$. This taming pair induces bases $\cU$ and $\cU'$ of open neighborhoods of $0$
on $K_S/\cM_S$ and $K_S'/\cM_{S'}$, respectively, as descibed in \S\ref{S:nuclear-taming}. It is easily checked that
$$\eta(\cU')=\cU, \qquad \psi^{-1}(\cU')=\cU_v:=\{\mathfrak{m}_v^i\cM_v\}_{i\geq 1}.$$
Obviously, $\cU_v$ defined above is a basis of open neighborhoods of $0$ for the compact, $G$--c.t. $R$--module $\cM_v$, satisfying the properties in Definition \ref{D:Umdef}.

Since $\Phi\in \OX_{F,S}\{\tau\}[[Z]]\tau Z$, the coefficients $\varphi_n$ of $\Phi$ are in $\cO_{F,S}\{\tau\}\tau$. Therefore, they all commute with $\psi$ and $\eta$ and are local contractions
with respect to $\cU_v$, $\cU'$ and $\cU$. (See Corollary \ref{C:tauPolyLocallyContracting}.) Consequently, we may apply Proposition \ref{MultSES} to obtain the following.
$$\det\nolimits_{R[[Z]]}(1 + \Phi\rvert \frac{K_{S'}}{\mathcal{M}_{S'}}) = \det\nolimits_{R[[Z]]}(1 + \Phi\rvert \mathcal{M}_v)\det\nolimits_{R[[Z]]}(1 + \Phi\rvert \frac{K_S}{\mathcal{M}_S}),$$
where the nuclear determinants above are computed with respect to $\cU'$, $\cU_v$ and $\cU$, respectively.
Since $\Phi\in \OX_{F,S}\{\tau\}[[Z]]\tau Z$ and $v\not\in S$, we may take $\mathfrak{m}_v\mathcal{M}_v$ as a common nucleus for all the coefficients $\varphi_n$ of $\Phi$, viewed as a nuclear operator on $\cM_v$. Then, since $\mathcal{M}_v/\mathfrak{m}_v\mathcal{M}_v\simeq \mathcal{M}/v\mathcal{M}$ as $R$--modules, we have
$$\det\nolimits_{R[[Z]]}(1 + \Phi\rvert \mathcal{M}_v) = \det\nolimits_{R[[Z]]}(1 + \Phi\rvert \mathcal{M}/v\mathcal{M}).$$
The last two displayed equalities give the desired result.
\end{proof}

\begin{theorem}\label{T:TraceFormula} (The Trace Formula) Let $\mathcal{M}$ be a taming module for $K/F$, and let $S$ be a finite set of primes of $F$ containing $S_\infty$. Let $\Phi\in \OX_{F,S}\{\tau\}[[Z]]\tau Z$. Then, we have
$$\prod_{v\in{\rm MSpec}(\cO_{F,S})}\det\nolimits_{R[[Z]]}(1 + \Phi\rvert \mathcal{M}/v\mathcal{M}) = \det\nolimits_{R[[Z]]}(1 + \Phi\rvert K_S/\mathcal{M}_S)^{-1}.$$
\end{theorem}
\begin{proof} Let $\ds\Phi = \sum_{n = 1}^\infty\varphi_nZ^n\in \OX_{F,S}\{\tau\}[[Z]] \tau Z$. We show that we have an equality
$$\prod_{v\in{\rm MSpec}(\cO_{F,S})}\det\nolimits_{R[[Z]]/Z^N}(1 + \Phi \rvert \mathcal{M}/v\mathcal{M})=\det\nolimits_{R[[Z]]/Z^N}(1 + \Phi\rvert K_S/\mathcal{M}_S)^{-1}$$
in $R[[Z]]/Z^N$. Then, the desired result follows by taking a projective limit, as $N\to\infty$.

Let $D = D_N$ be such that ${\rm deg}_\tau\varphi_n<\frac{nD}{N}$, for all $n<N$. Let
$$T := T_D := S\cup\{v\in\text{MSpec}(\cO_{F,S}) \mid  [\OX_{F,S}/v:\F_q]< D\}.$$
By Lemma \ref{LocalizationLemma}, it suffices to show that
$$\prod_{v\in\text{MSpec}(\cO_{F,T})}\det\nolimits_{R[[Z]]/Z^N}(1 + \Phi\rvert \mathcal{M}/v\mathcal{M})=\det\nolimits_{R[[Z]]/Z^N}(1 + \Phi\rvert K_T/\mathcal{M}_T)^{-1}.$$

Let $\mathcal{S}_{D,N}\subseteq \OX_{F,T}\{\tau\}[[Z]]/Z^N$ be the set
$$\mathcal{S}_{D,N} = \left\{1 + \sum_{n = 1}^{N-1}\psi_nZ^n \bigg\vert  \deg_{\tau}(\psi_n) < \frac{nD}{N}, \text{ for all } n<N\right\}.$$
The set $\mathcal{S}_{D,N}$ is a group under multiplication, and $(1 + \Phi)\bmod Z^N\in\mathcal{S}_{D,N}$. Now, following Taelman, we use a trick of Anderson (\cite[Prop 9]{Anderson00}). Since $\OX_{F,T}$ has no residue fields of degree $d<D$ over $\Fq$, for every $d<D$ there exists $f_{dj}, a_{dj}\in\OX_{F,T}$, with $1\leq j\leq M_d$, such that
$$1 = \sum_{j = 1}^{M_d} f_{dj}(a_{dj}^{q^d} - a_{dj}).$$
Then for every $r\in \OX_{F,T}$, and every $n<N$ and $d<D$, we have
$$1 - r\tau^dZ^n\equiv \prod_{j = 1}^{M_d}\frac{1 - (rf_{dj}\tau^d)a_{dj}Z^n}{1 - a_{dj}(rf_{dj}\tau^d)Z^n}\bmod Z^{n+1}.$$
Using this congruence it follows that the group $\mathcal{S}_{N,D}$ is generated by the set
$$\left\{\frac{1 - (s\tau^d)aZ^n}{1 - a(s\tau^d)Z^n}\mid a, s\in \cO_{F, T}, d, n\geq 1\right\}.$$
By properties of finite determinants, we have
$$\det\nolimits_{R[[Z]]/Z^N}\left(\frac{1 - (s\tau^d)aZ^n}{1 - a(s\tau^d)Z^n}\bigg\vert \mathcal{M}/v\mathcal{M}\right) = 1, \text{ for all }v\in\text{MSpec}(\cO_{F,T}).$$
Also, by Lemma \ref{CommDet}, we have
$$\det\nolimits_{R[[Z]]/Z^N}\left(\frac{1 - (s\tau^d)aZ^n}{1 - a(s\tau^d)Z^n}\bigg\rvert \frac{K_T}{\mathcal{M}_T}\right) = 1.$$
Consequently, Proposition \ref{MultDet} leads to the equalities
$$\det\nolimits_{R[[Z]]/Z^N}(1 + \Phi\rvert \frac{K_T}{\mathcal{M}_T}) = 1=\prod_{v\in{\rm MSpec}(\OX_{F,T})}\det\nolimits_{R[[Z]]/Z^N}(1 + \Phi\rvert \mathcal{M}/v\mathcal{M}),$$
which conclude the proof of the Theorem.
\end{proof}

\begin{corollary} \label{C:TraceFormula} Let $\mathcal{M}$ be a taming module for $K/F$. Let $E$ be a Drinfeld module with structural morphism $\varphi_E:\Fq[t]\to \cO_F\{\tau\}$. Then $\Phi = \frac{1 - \varphi_E(t)T^{-1}}{1 - tT^{-1}} - 1$ is a nuclear operator on $K_\infty/\mathcal{M}[[T^{-1}]]$ and
\[\Theta_{K/F}^{E,\mathcal{M}}(0) = \det\nolimits_{R[[T^{-1}]]}(1 + \Phi\mid K_\infty/\mathcal{M})\rvert_{T = t}.\]
\end{corollary}

\begin{proof} By Remark \ref{R:Nuclear-Fitting} applied to the free $R$--modules $\cM/v$ and $E(\cM/v)$ of the same rank $n_v:=[\cO_F/v:\Fq]$ and the definition of $\Theta_{K/F}^{E,\mathcal{M}}(0)$, we have
\[ \Theta_{K/F}^{E,\mathcal{M}}(0) = \prod_{v} \frac{|\cM/v|_G}{|E(\cM/v)|_G}=\prod_{v} \frac{\det\nolimits_{R[[T^{-1}]]}(1 - tT^{-1}\mid \cM/v)\rvert_{T = t}}{\det\nolimits_{R[[T^{-1}]]}(1 - \varphi_E(t)T^{-1}\mid \cM/v)\rvert_{T = t}},\]
where the products are taken over all $v\in{\rm MSpec}(\cO_F)$. Note that we have used the fact that $t$ acts as $\varphi_E(t)$ on $E(\cM/v)$. Since
\[\Phi = \sum_{n = 1}^\infty (t - \varphi_E(t))t^{n-1}T^{-n}\in \OX_{F}\{\tau\}[[T^{-1}]]\tau T^{-1},\]
by Corollary \ref{C:tauPolyLocallyContracting}, $\Phi$ is a nuclear operator on $K_\infty/\cM$ and $\cM/v$, for all $v$. Now, Theorem \ref{T:TraceFormula} applied in the case $S:=S_\infty$, combined with the previously displayed equalities gives
\[\Theta_{K/F}^{E,\mathcal{M}}(0) =\prod_{v}\det\nolimits_{R[[T^{-1}]]}(1 + \Phi\mid\cM/v)^{-1}\rvert_{T = t} = \det\nolimits_{R[[T^{-1}]]}(1 + \Phi\mid K_\infty/\cM)\rvert_{T = t},\]
which concludes the proof.
\end{proof}

\section{The volume function}\label{S:volume}
In this section we define the volume function ${\rm Vol}:\cC\to\Fq((t^{-1}))[G]^+$ on the class $\cC$ of compact $A[G]$--modules described
in Definition \ref{D:classC} with values in the subgroup $\Fq((t^{-1}))[G]^+$ of monic elements inside $\Fq((t^{-1}))[G]^\times$, defined in \S\ref{S:monic}.

\subsection{Indices of projective $A[G]$--lattices}\label{S:lattice-index} The first ingredient needed for defining the desired volume function is a notion of an index
$[\Lambda:\Lambda']_G\in \Fq((t^{-1}))[G]^+$, for any two projective $A[G]$--lattices $\Lambda, \Lambda'\subseteq K_\infty$. (See Definition \ref{D:lattices} for lattices.)

For the moment let us assume that $\Lambda$ and $\Lambda'$ are both free $A[G]$--lattices, of bases ${\bf e}:=(e_1, ..., e_n)$ and ${\bf e'}:=(e'_1, ..., e'_n)$,
where $n:=[F:\Fq(t)]$. Then, ${\bf e}$ and ${\bf e'}$ remain $\Fq((t^{-1}))[G]$--bases for $K_\infty$ (an immediate consequence of Definition \ref{D:lattices}). Therefore there exists a unique matrix  $X\in{\rm GL}_n(\Fq((t^{-1}))[G])$, such that
$({\bf e'})^t=X\cdot{\bf e}^t$. While the determinant ${\rm det}(X)$ depends on the choice of ${\bf e}$ and ${\bf e'}$, its image ${\rm det}(X)^+$ via the canonical  group morphism
$$\Fq((t^{-1}))[G]^\times\twoheadrightarrow\Fq((t^{-1}))[G]^\times/\Fq[t][G]^\times\simeq \Fq((t^{-1}))[G]^+$$
(see Corollary \ref{C:monic-rep}) obviously does not depend on any choices.
\begin{definition}\label{D:free-lattice-index}
For free $A[G]$--lattices $\Lambda, \Lambda'\subseteq K_\infty$, we define $[\Lambda:\Lambda']_G:={\rm det}(X)^+$.
\end{definition}
\begin{remark}\label{R:properties-free-index} If $\Lambda\subseteq\Lambda'$ are free $A[G]$--lattices in $K_\infty$,
then $\Lambda/\Lambda'$ is a finite, $G$--c.t. $A[G]$--module. From the definition of Fitting ideals, one can easily see that
$$[\Lambda:\Lambda']_G=|\Lambda/\Lambda'|_G.$$
Moreover, if $\Lambda''\subseteq\Lambda'\subseteq\Lambda$ are free $A[G]$--lattices in $K_\infty$, then
$$[\Lambda:\Lambda'']_G=[\Lambda:\Lambda']_G\cdot[\Lambda':\Lambda'']_G.$$
\end{remark}

The following Lemma permits us to transition from free to projective $A[G]$--lattices.
Recall that Definition \ref{D:G-size} associates to any finite, $G$-c.t. $A[G]$--module $M$ the unique
monic generator $|M|_G$ of ${\rm Fitt}^0_{A[G]}M$. This belongs to $\Fq[G][t]^+=\Fq[G][t]\cap\Fq((t^{-1}))[G]^+$.

\begin{lemma}\label{L:embed-projective-free}Let $\Lambda$ be a projective $A[G]$--lattice in $K_\infty$. Then
\begin{enumerate}\item  There exists a free $A[G]$--lattice $\cF$ of $K_\infty$, such that
$\Lambda\subseteq \cF$;
\item For any $\cF$ as above, the quotient $\cF/\Lambda$ is a finite, $G$--c.t. $A[G]$--module.
\item For any $\cF$ as above $|\cF/\Lambda|_G\in\Fq((t^{-1}))[G]^+$ is well defined.
\end{enumerate}
\end{lemma}
\begin{proof} (1)  follows from Proposition \ref{P:lattices-projective} in the Appendix. (2) follows from $\Fq(t)\Lambda=\Fq(t)\cF$ and the fact that both $\Lambda$ and $\cF$ are $G$--c.t.
(3) is a direct consequence of (2).
\end{proof}

\begin{definition}\label{D:projective-lattice-index}
Let $\Lambda$ and $\Lambda'$ be two projective $A[G]$--lattices in $K_\infty$. Let $\cF$ and $\cF'$ be free $A[G]$--lattices in $K_\infty$,
such that $\Lambda\subseteq\cF$ and $\Lambda'\subseteq \cF'$. Define
$$[\Lambda:\Lambda']_G:=[\cF:\cF']_G\cdot\frac{|\cF'/\Lambda'|_G}{|\cF/\Lambda|_G},$$
where $[\cF:\cF']_G$ is defined in \ref{D:free-lattice-index} above.
\end{definition}

\begin{lemma}\label{L:projective-lattice-index} With notations as in Definition \ref{D:projective-lattice-index}, we have the following:
\begin{enumerate}
\item $[\Lambda:\Lambda']_G$ is independent of the chosen $\cF$ and $\cF'$.
\item If $\Lambda, \Lambda', \Lambda''\subseteq K_\infty$ are projective $A[G]$--lattices, then
$$[\Lambda:\Lambda'']_G=[\Lambda:\Lambda']_G\cdot[\Lambda':\Lambda'']_G.$$
\item If $\Lambda'\subseteq\Lambda\subseteq K_\infty$ are projective $A[G]$--lattices, then
$$[\Lambda:\Lambda']_G=|\Lambda/\Lambda'|_G.$$
\end{enumerate}
\end{lemma}
\begin{proof} We prove (1) and leave the proofs of (2) and (3) to the interested reader. Since any two free $A[G]$--lattices $\cF_1$ and $\cF_2$ which
contain $\Lambda$ (respectively $\Lambda'$) can be embedded into a third free $A[G]$--lattice $\cF_3$ which contains $\Lambda$ (respectively $\Lambda'$),
it suffices to prove that
\begin{equation}\label{E:fractions}[\cF_1:\cF_1']_G\cdot \frac{|\cF_1'/\Lambda'|_G}{|\cF_1/\Lambda|_G}=  [\cF_2:\cF_2']_G\cdot \frac{|\cF_2'/\Lambda'|_G}{|\cF_2/\Lambda|_G},
\end{equation}
for any free $A[G]$--lattices $\cF_1, \cF_2, \cF_1', \cF_2'$,  such that $\Lambda\subseteq\cF_1\subseteq\cF_2$ and $\Lambda'\subseteq\cF_1'\subseteq\cF_2'$.
We have an obvious exact sequence of finite, $G$--c.t. $A[G]$--modules
$$0\to \cF_1/\Lambda\to\cF_2/\Lambda\to\cF_2/\cF_1\to 0.$$
Combined with Lemma \ref{L:FitSequence}, this yields the equality $|\cF_2/\Lambda|_G=|\cF_1/\Lambda|_G\cdot|\cF_1/\cF_2|_G$. Similarly, we obtain an equality
 $|\cF_2'/\Lambda'|_G=|\cF_1'/\Lambda'|_G\cdot|\cF_1'/\cF_2'|_G$. Now, the desired equality \eqref{E:fractions} follows from Remark \ref{R:properties-free-index} above.
\end{proof}

The index defined in \ref{D:projective-lattice-index} for projective $A[G]$--lattices restricts to Taelman's definition \cite{T12} of a (projective) $A$--lattice index when $G$ is trivial.

\subsection{The volume function and its properties}\label{S:volume-function} Let $M$ be an $A[G]$--module in the class $\cC$ described in Definition \ref{D:classC} of the Introduction. We refer to
\begin{equation}\label{E:classC-sequence}
0\to K_\infty/\Lambda\overset{\iota}{\longrightarrow}M\overset{\pi}{\longrightarrow} H\to 0
\end{equation}
as the structural exact sequence of topological $A[G]$--modules for $M$, where $\Lambda$ is an $A[G]$--lattice in $K_\infty$ and $H$ is a finite $A[G]$--module.
Recall that, by definition, $M$ is $G$--c.t. We let $[M]\in{\rm Ext}^1_{A[G]}(H, K_\infty/\Lambda)$ denote the extension class corresponding to \eqref{E:classC-sequence}.

Now, since $K_\infty/\Lambda$ is $A$--divisible, therefore $A$--injective (because $\Fq((t^{-1}))/A$ is), $\pi$ admits a section $s$ in the category of $A$--modules (not $A[G]$--modules, in general.)
Pick such a section $s$ and note that we have an $A$--module isomorphism
$$K_\infty/\Lambda\times s(H)\simeq M,$$
given by $(\iota, {\rm id})$. To simplify notation, we will drop $\iota$ from the notation and will think of it as an inclusion and of the isomorphism above as an equality in what follows.
\begin{definition}\label{D:Admissible lattice} An  $A[G]$--lattice $\Lambda'$ in $K_\infty$ is called $(M, s)$--admissible if
\begin{enumerate}
\item $\Lambda\subseteq\Lambda'$;
\item $\Lambda'$ is $A[G]$--projective;
\item $\Lambda'/\Lambda\times s(H)$ is an $A[G]$--submodule of $M$.
\end{enumerate}
An $A[G]$--lattice $\Lambda'$ is called $M$--admissible if it is $(M,s)$--admissible for some $s$.
\end{definition}

\begin{proposition}\label{P:Admissible}
For $(M, s)$ as above, there exist $A[G]$--free, $(M,s)$--admissible lattices.
\end{proposition}
\begin{proof}
Let $\widetilde\Lambda'$ be a free $A[G]$--lattice satisfying property (1) in the above definition. (See Proposition \ref{P:lattices-projective} in the Appendix for its existence.)
We will modify $\widetilde\Lambda'$ so that it will satisfy property (3) as well. For that, let $x\in H$ and $g\in G$, and let (under the $G$--action on $M$)
$$g\cdot (0,s(x)) := (a_{g,x},b_{g,x})\in(K_\infty/\Lambda\times s(H))=M.$$
Since the $A$--module $s(H)\simeq H$ is finite, there exists some $f'\in A\setminus\{0\}$ such that $f'\cdot s(x) = 0$, for all $x \in H$.  Then, since the $G$-action on $M$ commutes with multiplication by elements in $A$, we find that $f'a_{g,x} = 0$, for all $x$ and $g$ as above.  Now, it is easily seen that the free $A[G]$--lattice  $\Lambda' := \tfrac{1}{f'}\widetilde\Lambda'$ is $(M,s)$--admissible.
\end{proof}
\begin{remark}\label{R:Admissible} Note that if $M_i\in\cC$, $[M_i]\in{\rm Ext}^1_{A[G]}(H_i, K_\infty/\Lambda_i)$, for $i=1, \dots, m$, such that $\Fq(t)\Lambda_i$ is independent of $i$ (i.e. the $\Lambda_i$'s are contained in a common $A[G]$--lattice $\Lambda$), and $s_i$ is a fixed section for $M_i$, then the proof of the Proposition above can be easily adapted to show that there is a lattice $\Lambda'$ which is $(M_i, s_i)$--admissible, for all $i$.

Also, note that given data $(M,s)$ as above and an admissible $(M,s)$--lattice $\Lambda'$ we have a short exact sequence of $A[G]$-modules
\begin{equation}\label{E:admissible-sequence} 0 \rightarrow \Lambda'/\Lambda\times s(H)\rightarrow M \rightarrow K_\infty/\Lambda'\rightarrow 0.\end{equation}
Consequently, since $M$ is $G$--c.t. (by definition) and $K_\infty/\Lambda'$ is $G$--c.t. (because $K_\infty$ and $\Lambda'$ are),  $\Lambda'/\Lambda\times s(H)$ is a finite $A[G]$--module which is $G$--c.t.  Consequently, the monic element
$$|\Lambda'/\Lambda\times s(H)|_G\in \Fq((t^{-1}))[G]^+$$
is well defined, for any admissible $(M,s)$--lattice $\Lambda'$.
\end{remark}

Now, we are ready to define the desired volume function.  To make the definition, we first fix a projective $A[G]$--lattice $\Lambda_0\subseteq K_\infty$, which will be used for normalization. The volume function will
depend on $\Lambda_0$, but not in an essential way.
\begin{definition}\label{D:volume}
Let $M\in\cC$, with $[M]\in {\rm Ext}^1_{A[G]}(H, K_\infty/\Lambda)$. Let $s$ be a section for $M$ and $\Lambda'$ an $(M,s)$--admissible lattice. We define
\begin{equation}\label{D:Vol}
\Vol(M) = \frac{|\Lambda'/\Lambda \times s(H)|_G}{\left [\Lambda':\Lambda_0\right ]_{G}},
\end{equation}
where $[\Lambda':\Lambda_0]_G$ is as in Definition \ref{D:projective-lattice-index} and $|\Lambda'/\Lambda \times s(H)|_G$ is as
in Remark \ref{R:Admissible}.
\end{definition}

The next result shows that ${\rm Vol}$ is well defined, i.e. is independent of all choices except for $\Lambda_0$, and its dependence
of $\Lambda_0$ disappears in quotients.
\begin{proposition}
The function $\Vol:\cC\to \laurent{\F_q}{T\inv}[G]^+$ satisfies the following properties.
\begin{enumerate}
\item  For each $M \in \cC$ given by an exact sequence \eqref{E:classC-sequence}, the value $\Vol(M)$ is independent of  choice of section $s$ and of choice of
$(M,s)$--admissible lattice $\Lambda'$.
\item  $\Vol(K_\infty/\Lambda_0) = 1$.
\item  If $M_1,M_2 \in \cC$, the quantity $\frac{\Vol(M_1)}{\Vol(M_2)}$ is independent of choice of $\Lambda_0$.
\item If $M\in\cC$, with  $[M]\in {\rm Ext}^1_{A[G]}(H, K_\infty/\Lambda)$, then ${\rm Vol}(M)$ depends only on the extension class $[M]$ (if $H$ and $\Lambda$ are fixed.)
\end{enumerate}
\end{proposition}

\begin{proof} (1) First, let $s$ be a section for $M$ and let $\Lambda'$ and $\Lambda''$ be $(M,s)$--admissible lattices. Since $\Lambda\subset\Lambda', \Lambda''$, Remark \ref{R:Admissible} shows that we may assume without loss of generality that $\Lambda'\subseteq\Lambda''$. Then, we have a short exact sequence of finite $A[G]$--modules, which are $G$-c.t.
\[0\rightarrow \Lambda'/\Lambda \times s(H) \rightarrow \Lambda''/\Lambda \times s(H) \rightarrow \Lambda''/\Lambda' \rightarrow 0.\]
Applying Lemma \ref{L:FitSequence} in the Appendix to the above sequence gives an equality
$$|\Lambda''/\Lambda\times s(H)|_{G}=|\Lambda'/\Lambda\times s(H)|_{G}\cdot |\Lambda''/\Lambda'|_{G} =|\Lambda'/\Lambda\times s(H)|_{G}\cdot [\Lambda'':\Lambda']_{G}.$$
Independence on $\Lambda'$ follows from the equality above combined with Lemma \ref{L:projective-lattice-index}(2).

Now, for two distinct sections $s_1$ and $s_2$, it is easy to see that one can pick a sufficiently large lattice $\Lambda'$ which is both $(M, s_1)$-- and $(M, s_2)$--admissible, and with the additional property that for all $x\in H$,
$(s_1(x)-s_2(x))\in \Lambda'/\Lambda$. It is easily seen that for such $\Lambda'$, the identity map on $M$ induces an isomorphism of $A[G]$--modules
$$\Lambda'/\Lambda\times s_1(H)\simeq \Lambda'/\Lambda\times s_2(H).$$
Therefore $|\Lambda'/\Lambda\times s_1(H)|_G=|\Lambda'/\Lambda\times s_2(H)|_G$, which proves independence on $s$.

 Part (2) is immediate as $\Lambda_0$ is $K_\infty/\Lambda_0$--admissible.

Part (3) follows by noting that for $M_1, M_2\in\cC$, we have
$$\frac{\Vol(M_1)}{\Vol(M_2)} = \frac{|\Lambda_1'/\Lambda_1 \times H_1|_{G}}{|\Lambda_2'/\Lambda_2 \times H_2|_{G}}\cdot [\Lambda_2':\Lambda_1']_{G},$$
where the notations are the obvious ones.

Part (4) is left to the interested reader, as it will not be used in this paper.
\end{proof}

\section{A $G$-equivariant volume formula}\label{S:VolumeFormula}
The purpose of this section is to express determinants of certain nuclear operators in the sense of \S \ref{S:Nuclear} in terms of a quotient of volumes in the sense of \S \ref{S:volume}.  Eventually, this will allow us to express our special $L$-values $\Theta_{K/F}^{E, \cM}(0)$ in terms of volumes, in preparation for proving the ETNF and the Drinfeld module analogue of the refined Brumer-Stark conjecture.

\subsection{Maps tangent to the identity}\label{S:Ntang} Below, $K_\infty$ is endowed with the sup of the local norms, denoted $||\cdot||$, normalized so that $\lVert t\rVert=q.$ The closed unit ball in $K_\infty$ is denoted $\cO_{K_\infty}$, as usual.

Let $M_1,M_2\in \cC$ of structural short exact sequences
$$0\to K_\infty/\Lambda_s\overset{\iota_s}\longrightarrow M_s\overset{\pi_s}\longrightarrow H_s\to 0, \qquad s=1,2.$$
Fix $\ell>0$ sufficiently large so that $t^{-i}\cO_{K_\infty}\cap\Lambda_s=\{0\}$, for all $i\geq\ell$ and $s=1, 2$ and identify $t^{-i}\cO_{K_\infty}$ with its image in $K_\infty/\Lambda_s$, for all $i\geq \ell$. Fix an $\infty$--taming module $\mathcal W^\infty$ for $K/F$. With notations as in \S\ref{S:nuclear-classC}, the resulting $\{\iota_s(U_{i, \infty})\}_{i\geq \ell}$ are appropriate bases of $R$--projective, open neighborhoods of $0$
in $M_s$, for all $s=1,2$. By \eqref{E:taming-open} there exists $a\in\Z_{>0}$, which we fix once and for all, such that
\begin{equation}\label{E:U-inclusions} t^{-a-i}\cO_{K_\infty}\subseteq U_{i, \infty}\subseteq t^{-i}\cO_{K_\infty}, \qquad\text{ for all }i\geq\ell.
\end{equation}
We endow $\iota_s(t^{-i}\cO_{K_\infty})$ with the norm which makes
$\iota_s: t^{-i}\cO_{K_\infty}\simeq \iota_s(t^{-i}\cO_{K_\infty})$ bijective isometries, for all $s=1, 2$, and all $i\geq \ell$.

\begin{definition}\label{D:N-tangent} Let $N\in\Z_{\geq 0}$. A continuous $R$--module morphism $\gamma:M_1\to M_2$ is called $N$-tangent to the identity if there exists $i\geq \ell$ such that
\begin{enumerate}
\item $\gamma$ induces a bijective isometry $(\iota_2^{-1}\circ\gamma\circ\iota_1): t^{-i}\cO_{K_\infty}\simeq t^{-i}\cO_{K_\infty}.$
\item If we let $\gamma_i$ denote the bijective isometry $(\iota_2^{-1}\circ\gamma\circ\iota_1): t^{-i}\cO_{K_\infty}\simeq t^{-i}\cO_{K_\infty}$, then
$$||\gamma_i(x)-x||\leq ||t||^{-N-a}\cdot||x||, \qquad\text{ for all }x\in t^{-i}\cO_{K_\infty}.$$
\end{enumerate}
If $\gamma$ is $N$--tangent to the identity for all $N\geq 0$, $\gamma$ is called infinitely tangent to the identity.
\end{definition}

\begin{proposition}\label{P:Ntangpowerseries}
 Let $\Gamma:K_\infty\to K_\infty$ be an $R$-linear map given by an everywhere convergent power series
\[\Gamma(z) = z + \alpha_1 z^q + \alpha_2 z^{q^2} + \dots,\qquad\text{with } \alpha_i \in K_\infty.\]
Assume that $\Gamma(\Lambda_1) \subseteq \Lambda_2$ and denote by $\widetilde{\Gamma}:K_\infty/\Lambda_1\to K_\infty/\Lambda_2$ the induced map. Assume that $\gamma:M_1\to M_2$ is a continuous $R$--linear morphism such that $\iota_2^{-1}\circ\gamma\circ\iota_1=\widetilde\Gamma$ on $t^{-\ell}\cO_{K_\infty}$.  Then $\gamma$ is infinitely tangent to the identity.
\end{proposition}
\begin{proof}
Let $N\geq 1$. We will show that $\gamma$ is $N$--tangent to the identity.
Since the power series for $\Gamma$ is everywhere convergent, the coefficients $\alpha_i$ must be bounded in norm.  Let $\alpha:=\sup_i||\alpha_i||.$
Thus, if $i\geq \ell$ is sufficiently large and $z\in t^{-i}\cO_{K_\infty}$, then we have
$$\lVert(\iota_2^{-1}\circ\gamma\circ\iota_1)(z)\rVert=\lVert z\rVert, \quad \lVert(\iota_2^{-1}\circ\gamma\circ\iota_1)(z) - z\rVert = \lVert(\alpha_1z^q + \alpha_2 z^{q^2} + \cdots)\rVert\leq \alpha\cdot||z||^q.$$
In particular, if $i$ is sufficiently large, then $(\iota_2^{-1}\circ\gamma\circ\iota_1): t^{-i}\cO_{K_\infty}\to t^{-i}\cO_{K_\infty}$ is an isometry, which is strictly
differentiable at $0$ and $(\iota_2^{-1}\circ\gamma\circ\iota_1)'(0)=1$. By the non-archimedean inverse function theorem (see \cite[2.2]{Igusa}), for all $i\gg\ell$ the map $(\iota_2^{-1}\circ\gamma\circ\iota_1): t^{-i}\cO_{K_\infty}\simeq t^{-i}\cO_{K_\infty}$ is a bijective isometry.
Further, for all $i\gg \ell$ and all $z\in t^{-i}\cO_{K_\infty}\setminus \{0\}$, we have
$$\frac{\lVert(\iota_2^{-1}\circ\gamma\circ\iota_1)(z) - z\rVert}{\lVert z\rVert}\leq \alpha\lVert z\rVert^{q-1}\leq \alpha\lVert t\rVert^{-i(q-1)}\leq\lVert t\rVert^{-N-a},$$
which shows that, indeed, $\gamma$ is $N$--tangent to the identity.
\end{proof}

\begin{definition} Let $M_1, M_2\in\mathcal C$ and let $\gamma:M_1\simeq  M_2$ be an $R$--linear topological isomorphism. We define the endomorphism $\Delta_\gamma$ of $\power{M_1}{T\inv}$ by
$$\Delta_\gamma := \frac{1 - \gamma\inv t \gamma T\inv}{1-tT\inv} - 1=\sum_{n=1}^\infty \delta_n T^{-n},$$
where $\delta_n = (t-\gamma\inv t \gamma )t^{n-1}$, for all $n\geq 1.$
\end{definition}

\begin{lemma}\label{L:Ntangnuclear}
If the topological $R$--linear isomorphism $\gamma:M_1\simeq M_2$ is $N$-tangent to the identity, then the map $(\Delta_\gamma$ mod $T^{-N})$ is a nuclear endomorphism of $\power{M_1}{T\inv}/T^{-N}$. If $\gamma$ is infinitely tangent to the identity, then $(1+\Delta_\gamma)$ is a nuclear endomorphism of $M_1[[T^{-1}]]$.
\end{lemma}
\begin{proof} For simplicity, below we suppress $\iota_1$ and $\iota_2$ from the notations (and think of them as inclusions.)
We need to show that each $\delta_n$ is locally contracting in the sense of \ref{D:Nuclear}, for all $n<N$.  Fix $n<N$, and fix $i\geq \ell$ as in Definition \ref{D:N-tangent} applied to $\gamma$. We will show that
$$\delta_n(U_{j, \infty})\subseteq U_{j+1,\infty}, \qquad\text{ for all }j\geq i+n.$$
Since $-j+n\leq -i$, we obviously have
$$\delta_n(t^{-j}\cO_{K_\infty})\subseteq t^{-j+n}\cO_{K_\infty}.$$
Also, if $\gamma_i$ is as in Definition \ref{D:N-tangent}, then we have equalities of functions defined on  $t^{-j}\cO_{K_\infty}$
$$\delta_{n}=(t-\gamma_i^{-1}t\gamma_i)t^{n-1}=\gamma_i^{-1}(\gamma_i-1)t^n+\gamma_i^{-1}t(1-\gamma_i)t^{n-1}.$$
Consequently, the conditions imposed upon $\gamma_i$ in Definition \ref{D:N-tangent} imply that
$$||\delta_{n}(z)||\leq ||t||^{-N+n-a}\cdot||z||\leq ||t||^{-1-a}\cdot||z||, \qquad\text{ for all }z\in t^{-j}\cO_{K_\infty}.$$
In particular, if $z\in U_{j,\infty}$ then $\delta_{n}(z)\in t^{-j-1-a}\cO_{K_\infty}$, and the inclusions \eqref{E:U-inclusions} show that $\delta_{n}(z)\in U_{j+1, \infty}.$
\end{proof}

\subsection{Endomorphisms of $K_\infty/\Lambda$.}\label{S:endomorphisms} Now, we treat the particular case $M_1=M_2=K_\infty/\Lambda$, for an $A[G]$--projective lattice $\Lambda\subseteq K_\infty$. As above, we fix
$\ell>0$ such that $t^{-\ell}\cO_{K_\infty}\cap\Lambda=\{0\}$ and fix $a\in\Z_{>0}$ satisfying \eqref{E:U-inclusions}. For simplicity, we let $V:=K_\infty/\Lambda$.
\begin{definition}\label{D:contraction}
An $R$--linear, continuous endomorphism $\phi: V\to V$ is called a local $M$--contraction, for some
$M\in\Bbb Z_{>0}$, if there exists $i\geq\ell$ such that
$$||\phi(x)||\leq ||t||^{-M}\cdot||x||, \text{ for all }x\in t^{-i}\cO_{K_\infty}.$$
\end{definition}
\begin{remark}\label{R:contraction-nuclear}
If $\phi$ as above is a local $M$--contraction for some $M>a$, then $\phi$ is locally contracting on $V$ and therefore
the nuclear determinant ${\rm det}_{R[[Z]]}(1-\phi\cdot Z|V)$ makes sense.  Indeed, pick an $i>\ell$ as in the definition above.
Then, inclusions \ref{E:U-inclusions} show that
$$\phi(U_{j,\infty})\subseteq\phi(t^{-j}\cO_{K_\infty})\subseteq t^{-j-M}\cO_{K_\infty}\subseteq t^{-j-a-1}\cO_{K_\infty}\subseteq U_{j+1, \infty},$$
for all $j\geq i$. This shows that $U_{i, \infty}$ is a nucleus for $\phi$.
\end{remark}
\begin{proposition}\label{P:alpha-psi}
Assume that $\gamma:V\simeq V$ is an $R$--linear, continuous isomorphism, which
is $N$--tangent to the identity, for some $N>0$. Let $\psi:V\to V$ be an $R$--linear, continuous, local $M$--contraction, for some
$M>2a$. Let $\alpha:=t\gamma$. Then
\begin{enumerate}
\item $\alpha\psi$ and $\psi\alpha$ are local $(M-1)$--contractions on $V$.
\item ${\rm det}_{R[[Z]]}(1-\alpha\psi\cdot Z|V)={\rm det}_{R[[Z]]}(1-\psi\alpha\cdot Z|V).$
\end{enumerate}
\end{proposition}
\begin{proof}  Fix $i>\ell$ such that $\gamma:t^{-(i-1)}\cO_{K_\infty}\to t^{-(i-1)}\cO_{K_\infty}$ is a bijective isometry and such that
$$||\psi(x)||\leq ||t||^{-M}\cdot ||x||, \text{ for all }x\in t^{-(i-1)}\cO_{K_\infty}.$$

(1) For $i$ chosen as above it is easy to check that
$$||\alpha\psi(x)||\leq ||t||^{-(M-1)}||x||, \quad ||\psi\alpha(x)||\leq ||t||^{-(M-1)}||x||, \text{ for all }x\in t^{-i}\cO_{K_\infty}.$$
So, $\alpha\psi$ and $\psi\alpha$ are $(M-1)$--contractions on $t^{-i}\cO_{K_\infty}$. Therefore, they are locally contracting endomorphisms of $V$
by Remark \ref{R:contraction-nuclear}, and so the nuclear determinants in (2) make sense.
\medskip

(2) The last displayed inequalities, combined with $(M-1)>a$ and Remark \ref{R:contraction-nuclear} show that $\psi$, $\alpha\psi$, and  $\psi\alpha$
are all locally contracting on $V$ of common nuclei $U_{j, \infty}$, for all $j\geq i$.

Now, since $\gamma$ is an isomorphism and $V$ is $t$--divisible (because $K_\infty$ is), $\alpha$ is surjective. Therefore $\alpha$ induces and $R$--module isomorphism
$$V/\alpha^{-1}(U_{i,\infty})\overset{\alpha}\simeq V/U_{i, \infty}.$$
Since $\Lambda\cap U_{i, \infty}=\{0\}$, we have $\alpha^{-1}(U_{i, \infty})=\gamma^{-1}(\frac{1}{t}\Lambda/\Lambda)\oplus\gamma^{-1}(t^{-1}U_{i,\infty}).$ Below, we let
$$\alpha^{-1}(U_{i, \infty})^\ast:=\gamma^{-1}(t^{-1}U_{i,\infty}).$$
Since $\gamma$ is an isomorphism, the $R$--modules $\gamma^{-1}(t^{-1}U_{i,\infty})$, $\gamma^{-1}(\frac{1}{t}\Lambda/\Lambda)$ and $\alpha^{-1}(U_{i, \infty})$ are all projective
(equivalently, $G$--c.t.) because $U_{i,\infty}$ and $\Lambda$ are $G$--c.t. Also, note that
\begin{equation}\label{E:alpha-plus}t^{-(i+1)-a}\cO_{K_\infty}\subseteq \alpha^{-1}(U_{i, \infty})^\ast,\, U_{i+1, \infty}\subseteq t^{-(i+1)}\cO_{K_\infty},\end{equation}
as $\gamma^{-1}$ is an isometry on $t^{-(i+1)}\cO_{K_\infty}$ and $t^{-(i+1)-a}\cO_{K_\infty}\subseteq t^{-1}U_{i, \infty}\subseteq t^{-(i+1)}\cO_{K_\infty}$.
Now, use \eqref{E:alpha-plus} to note that since $\psi$ is an $M$--contraction on $t^{-i}\cO_{K_\infty}$ and $M>2a$, we have
\begin{equation}\label{E:psi-alpha}(\psi\alpha)(\alpha^{-1}(U_{i, \infty}))=\psi(U_{i, \infty})\subseteq t^{-i-M}\cO_{K_\infty}\subseteq t^{-(i+1)-a}\cO_{K_\infty}\subseteq \alpha^{-1}(U_{i, \infty})^\ast\subseteq \alpha^{-1}(U_{i, \infty}).
\end{equation}
Consequently, we have a commutative diagram of morphisms of finite, projective $R$--modules
$$\xymatrix{
V/\alpha^{-1}(U_{i, \infty})\ar@{>}[r]^{\quad \alpha}_{\quad \sim}\ar@{>}[d]^{\,\psi\alpha} & V/U_{i, \infty}\ar@{>}[d]^{\,\alpha\psi} \\
V/\alpha^{-1}(U_{i, \infty})\ar@{>}[r]^{\quad \alpha}_{\quad \sim} &V/U_{i, \infty} ,}$$
whose horizontal maps are isomorphisms. This gives an equality of (regular) determinants
\begin{equation}\label{E:det2}{\rm det}_{R[[Z]]}(1-\alpha\psi\cdot Z|V/U_{i, \infty})={\rm det}_{R[[Z]]}(1-\psi\alpha\cdot Z|V/\alpha^{-1}(U_{i, \infty})).\end{equation}
Now, consider the short exact sequence of projective $R$--modules
$$0\to {\alpha^{-1}(U_{i, \infty})}/{\alpha^{-1}(U_{i, \infty})^\ast}\to V/{\alpha^{-1}(U_{i, \infty})^\ast}\to V/{\alpha^{-1}(U_{i, \infty})}\to 0.$$
Noting that \eqref{E:psi-alpha} implies that $\psi\alpha$ induces an $R$--linear endomorphism of the exact sequence above and that $\psi\alpha\equiv 0$ on ${\alpha^{-1}(U_{i, \infty})}/{\alpha^{-1}(U_{i, \infty})^\ast}$,
the exact sequence above gives
\begin{equation}\label{E:det3}{\rm det}_{R[[Z]]}(1-\psi\alpha\cdot Z|V/\alpha^{-1}(U_{i, \infty}))={\rm det}_{R[[Z]]}(1-\psi\alpha\cdot Z|V/\alpha^{-1}(U_{i, \infty})^\ast).\end{equation}
Now, since $\psi\alpha$ is an $(M-1)$--contraction on $t^{-i}\cO_{K_\infty}$ (see proof of part (1)), \eqref{E:alpha-plus} leads to the following inclusions
$$\psi\alpha(\alpha^{-1}(U_{i, \infty})^\ast),\, \psi\alpha(U_{i+1,\infty})\subseteq t^{-(i+1)-(M-1)}\cO_{K_\infty}\subseteq t^{-2a-(i+1)}\cO_{K_\infty}\subseteq t^{-a}\left(\alpha^{-1}(U_{i, \infty})^\ast\right).$$
Now, since \eqref{E:alpha-plus} also implies that
$$t^{-a}\left(\alpha^{-1}(U_{i, \infty})^\ast\right)\subseteq U_{i+1,\infty},\, \alpha^{-1}(U_{i, \infty})^\ast,$$
the last displayed inclusions show that $\psi\alpha\equiv 0$ on the quotients $U_{i+1,\infty}/t^{-a}\left(\alpha^{-1}(U_{i, \infty})^\ast\right)$ and on $\alpha^{-1}(U_{i, \infty})^\ast/t^{-a}\left(\alpha^{-1}(U_{i, \infty})^\ast\right)$. Consequently, a short exact sequence argument similar to the one used to prove \eqref{E:det3} above gives the following equalities of (regular) determinants
\begin{eqnarray*}
 {\rm det}_{R[[Z]]}(1-\psi\alpha\cdot Z|V/\alpha^{-1}(U_{i, \infty})^\ast)  &=& {\rm det}_{R[[Z]]}(1-\psi\alpha\cdot Z|V/t^{-a}\alpha^{-1}(U_{i, \infty})^\ast) \\
  &=& {\rm det}_{R[[Z]]}(1-\psi\alpha\cdot Z|V/U_{i+1,\infty}).
\end{eqnarray*}
Now, we combine these equalities with \eqref{E:det2} and \eqref{E:det3} to obtain
$${\rm det}_{R[[Z]]}(1-\alpha\psi\cdot Z|V/U_{i, \infty})={\rm det}_{R[[Z]]}(1-\psi\alpha\cdot Z|V/U_{i+1,\infty}).$$
Recalling that $U_{i, \infty}$ and $U_{i+1, \infty}$ are common nuclei for $\psi\alpha$ and $\alpha\psi$, this leads to the desired equality
of nuclear determinants, which concludes the proof of part (2).
\end{proof}

\begin{remark}\label{R:monomials} Assume that $\psi$, $\gamma$, $\alpha$ and $M$ are as in Proposition \ref{P:alpha-psi}, however here $M$ can be any positive integer. An argument similar to that used in the proof of part (1) of Prop. \ref{P:alpha-psi} shows that any element in the $R$--subalgebra $R\{\alpha, \psi\}$ of ${\rm End}_R(V)$ generated by $\alpha$ and $\psi$ with the property that it is a sum of monomials of degree at most $n$, for some $n\leq M$, each containing at least one factor of $\psi$, is a local $(M-n+1)$--contraction on $V$. Examples of such monomials are $\alpha\psi$ and $\psi\alpha$, dealt with in Proposition \ref{P:alpha-psi}(1). We leave the details of the general case to the reader.
\end{remark}

\begin{corollary}\label{C:det=1}
Let $\gamma:V\simeq V$ be an $R$--linear, continuous isomorphism which is $(2N)$--tangent to the identity, for some $N\geq a$. Then, we have
$${\rm det}_{R[[T^{-1}]]/T^{-N}}(1+\Delta_\gamma\,|\, V[[T^{-1}]]/T^{-N})=1.$$
\end{corollary}
\begin{proof} We use the main ideas in the proof of Corollary 1 in \cite{T12}. Let $Z:=T^{-1}$.
Let $\alpha:=t\gamma$ and $\psi:=(\gamma^{-1}-1)$, viewed as a continuous, $R$--linear endomorphism of $V$. Then, we have
$$1+\Delta_\gamma=\frac{1-(\psi+1)\alpha\cdot Z}{1-\alpha(\psi+1)\cdot Z}.$$
Now, since $\gamma^{-1}$ is $(2N)$--tangent to the identity, $\psi$ is a local $(2N+a)$--contraction. (See Definition \ref{D:N-tangent}(2).)
As in the proof of Cor. 1 \cite{T12}, one writes
$$\frac{1-(\psi+1)\alpha\cdot Z}{1-\alpha(\psi+1)\cdot Z}\mod Z^N=\prod_{n=1}^{N-1}\left(\frac{1-\psi_n\alpha\cdot Z^n}{1-\alpha\psi_n\cdot Z^n}\right)\mod Z^N,$$
where the $\psi_n$'s are uniquely determined polynomials in $R\{\alpha, \psi\}$ of degree at most $n$, containing at least one factor of $\psi$.
According to Remark \ref{R:monomials}, $\psi_n$ is a local $(N+a+1)$--contraction on $V$, for all $n<N$.
Since $M:=(N+a+1)>2a$, we may apply Proposition \ref{P:alpha-psi}(2) to $\alpha$, $\psi:=\psi_n$ and $M$ to conclude that
$${\rm det}_{R[[Z]]/Z^N}(1+\Delta_\gamma\,|\,V[[Z]]/Z^N)=\prod_{n=1}^{N-1}{\rm det}_{R[[Z]]/Z^N}\left(\frac{1-\psi_n\alpha\cdot Z^n}{1-\alpha\psi_n\cdot Z^n}\,\bigg|\,V[[Z]]/Z^N\right)=1.$$
\end{proof}
\subsection{Volume interpretation of determinants.} The next theorem is motivated by the fact that if $\gamma:H_1\simeq H_2$ is an $R$--linear isomorphism of {\it finite}, projective $R[t]$--modules, then
\begin{equation}\label{E:finite-determinants-volumes}\det\nolimits_{\power{R}{T\inv}}\left (1 + \Delta_\gamma \big|H_1\right )\bigg\vert_{T=t} = \frac{|H_2|_G}{|H_1|_G}.\end{equation}
This follows immediately from Remark \ref{R:Nuclear-Fitting} and the observation that $H_1$ endowed with the modified $t$--action
$t\ast x=\gamma^{-1} t\gamma(x)$ is $R[t]$--isomorphic to $H_2$. ($\gamma$ gives an isomorphism.)

\begin{theorem}\label{T:Volume}
Let $M_1$ and $M_2$ be modules from the class $\cC$, and let $\gamma:M_1\simeq M_2$ be an $R$-linear, continuous isomorphism which is infinitely tangent to the identity. Further, assume that $M_2=K_\infty/\Lambda_2$, for a projective $R[t]$--lattice $\Lambda_2$ in $K_\infty$. Then
\[\det\nolimits_{\power{R}{T\inv}}\left (1 + \Delta_\gamma \big|M_1\right )\bigg\vert_{T=t} = \frac{\Vol(M_2)}{\Vol(M_1)}.\]
\end{theorem}

\begin{remark} Although we believe that the above Theorem holds for general $M_1$ and $M_2$, for the purposes of this paper
it is sufficient to prove this result for $M_2$ of the special type described above. We plan on addressing the general case in an
upcoming paper.
\end{remark}

\begin{proof}[Proof of Theorem \ref{T:Volume}] The proof follows the strategy in \S4 of \cite{T12}. Below, we use the notations in \S\S\ref{S:Ntang}--\ref{S:endomorphisms}. For simplicity, we suppress $\iota_1$ from the notations, and think of it as an inclusion. Recall that $R:=\Fq[G]$. We need two intermediate Lemmas.

\begin{lemma}[Independence of $\gamma$] \label{L:independencegamma}
For $M_1$ and $M_2$ as in Theorem \ref{T:Volume}, assume that $\gamma_1,\gamma_2:M_1\simeq M_2$ are two $R$-linear, continuous isomorphisms which are $2N$-tangent to the identity, for some $N\geq a$.  Then
\[\det\nolimits_{R[[T\inv]]/T^{-N}}(1+\Delta_{\gamma_1}|M_1) = \det\nolimits_{R[[T\inv]]/T^{-N}}(1+\Delta_{\gamma_2}|M_1).\]
\end{lemma}
\begin{proof} It is straightforward to see that if an $R$--linear morphism $\delta:M_2\to M_2$ is locally contracting and if an $R$--linear isomorphism $\gamma:M_1\simeq M_2$ is $2N$--tangent to the identity for $N\geq a$, then $\gamma^{-1}\delta\gamma:M_1\to M_1$ is locally contracting. This is a direct consequence of the definitions and the identity
$$\delta-\gamma^{-1}\delta\gamma=(1-\gamma^{-1})\delta+\gamma^{-1}\delta(1-\gamma).$$
In our context, this observation allows us to write
\begin{equation}\label{E:composition}(1+\Delta_{\gamma_1})=[\gamma_2^{-1}(1+\Delta_{\gamma_1\gamma_2^{-1}})\gamma_2]\cdot(1+\Delta_{\gamma_2}),\end{equation}
where all operators inside parentheses are nuclear mod $T^{-N}$. Hence, it suffices to show that
$$\det\nolimits_{R[[T\inv]]/T^{-N}}(1+\Delta_{\gamma_1\gamma_2^{-1}}\,|\,M_2)=1.$$
This follows directly from Corollary \ref{C:det=1} applied to $V:=M_2$ and $\gamma:=\gamma_1\gamma_2^{-1}$.
 \end{proof}

\begin{lemma}[Common over-lattice]\label{L:common-over-lattice}
For $M_1$ and $M_2$ as in Theorem \ref{T:Volume}, assume that the $A[G]$--lattices $\Lambda_1$ and $\Lambda_2$ are contained in a common
$A[G]$--lattice $\Lambda$ of $K_\infty$. Let $\gamma:M_1\simeq M_2$ be an $R$--linear isomorphism, which is $2N$--tangent to the identity,
for some $N>a$. Then
$${\rm det}_{R[[T^{-1}]]/T^{-N}}(1+\Delta_\gamma\,|\, M_1)|_{T=t}\equiv \frac{{\rm Vol}(M_2)}{{\rm Vol}(M_1)}\mod t^{-N}.$$
\end{lemma}
\begin{proof} Fix an $A$--linear section $s_1$ for $\pi_1$. Per Remark \ref{R:Admissible}, we may assume that $\Lambda$ is $(M_1, s_1)$--admissible. Hence,
$\Lambda$ is $A[G]$--free and $(\Lambda/\Lambda_1\times s_1(H_1))$ is a finite, projective $A[G]$--submodule of $M_1$. Now, we can pick an $R$--projective, open
submodule $\cU$ of $K_\infty$, such that
$$K_\infty=\Lambda\oplus\cU,$$
as $R$--modules. Indeed, if $\mathbf{e}=\{e_1, \dots, e_n\}$ is a $\Bbb F_q[t][G]$--basis for $\Lambda$, then $\mathbf{e}$ is an $\Fq((t^{-1}))[G]$--basis for $K_\infty$, so we let $\cU:=\bigoplus_{i=1}^n t^{-1}\Fq[[t^{-1}]][G]e_i$, which satisfies all the desired properties. Now, $\gamma$ gives an $R$--linear isomorphism
$$\gamma: M_1=\cU\oplus(\Lambda/\Lambda_1\times s_1(H_1))\simeq  \cU\oplus \Lambda/\Lambda_2=M_2,$$
where the two direct sums are viewed in the category of topological $R$--modules (not $R[t]$--modules, as $\cU$ is not an $R[t]$--submodule of $K_\infty$.) We claim that this implies
that there exists an $R$--module isomorphism (not necessarily induced by $\gamma$)
\begin{equation}\label{E:iso-xi} \xi: (\Lambda/\Lambda_1\times s_1(H_1))\simeq \Lambda/\Lambda_2.\end{equation}
To prove this, let us pick an $i\in\Z_{>\ell}$ sufficiently large, so that   $\gamma: t^{-i}\cO_{K_\infty}\to t^{-i}\cO_{K_\infty}$ is a bijective isometry and $t^{-i}\cO_{K_\infty}\subseteq \cU$ . Then $\gamma$
 induces an isomorphism of finite $R$--modules
 $$\gamma: S\oplus A_1  \simeq  S\oplus A_2,$$
 where $S:=\cU/t^{-i}\cO_{K_\infty}$, $A_1:=(\Lambda/\Lambda_1\times s_1(H_1))$ and $A_2=\Lambda/\Lambda_2$. Now, $R$ is a finite, semilocal ring. Let us split it into the direct sum
 $R:=\oplus_j R_j$ of its local components, as in \eqref{E:character-decomposition} and do the same for any $R$--module $M$, i.e. write $M=\oplus_j M_j$, where $M_j:=M\otimes_R R_j$. Obviously, $\gamma$ induces $R_j$--module isomorphisms
 $$\gamma_j: S_j\oplus (A_1)_j\simeq S_j\oplus (A_2)_j,$$
 for all $j$. Now, since all modules involved are finite, we must have an equality of cardinalities $|(A_1)_j|=|(A_2)_j|$, for all $j$. However, the modules $(A_1)_j$ and $(A_2)_j$ are $R_j$--projective, therefore $R_j$--free. Hence, since the rings $R_j$ are finite, the equality of cardinalities implies an equality of $R_j$--ranks, which in turn gives isomorphisms $(A_1)_j\simeq (A_2)_j$ as $R_j$--modules, for all $j$. Consequently, we have an isomorphism of $R$--modules $A_1\simeq A_2$, as desired.

 Fix an isomorphism $\xi$ as above and define the $R$--module isomorphism
 $$\rho: M_1=(\cU\oplus A_1)\simeq (\cU\oplus A_2)=M_2, \qquad \rho|_{\cU}={\rm id}_{\cU}\text{ and }\rho|_{A_1}=\xi.$$
Obviously, $\rho$ is infinitely tangent to the identity. Therefore, Lemma \ref{L:independencegamma} implies that
$${\rm det}_{R[[T^{-1}]]/T^{-N}}(1+\Delta_\gamma\,|\, M_1)={\rm det}_{R[[T^{-1}]]/T^{-N}}(1+\Delta_\rho\,|\, M_1).$$
Now, we have a commutative diagram of topological morphisms of modules in class $\mathcal C$
$$\xymatrix{
0\ar[r] & A_1\ar[r]\ar[d]^{\xi}_{\wr} & M_1\ar[r]\ar[d]^{\rho}_{\wr} & K_\infty/\Lambda\ar[r]\ar[d]^{{\rm id}}_{=} & 0 \\
0\ar[r] & A_2\ar[r] & M_2\ar[r] & K_\infty/\Lambda\ar[r] & 0,}$$
whose rows are exact and $R[t]$--linear (see \eqref{E:admissible-sequence}) and whose vertical maps are $R$--linear isomorphisms, $(2N)$--tangent to the identity. This leads to an equality of nuclear determinants
\begin{eqnarray*}
{\rm det}_{R[[T^{-1}]]/T^{-N}}(1+\Delta_\rho\,|\, M_1) &=& {\rm det}_{R[[T^{-1}]]/T^{-N}}(1+\Delta_\xi\,|\, A_1)\cdot {\rm det}_{R[[T^{-1}]]/T^{-N}}(1+\Delta_{\rm id}\,|\, K_\infty/\Lambda) \\
 &=&{\rm det }_{R[[T^{-1}]]/T^{-N}}(1+\Delta_\xi\,|\, A_1)
\end{eqnarray*}
However, \eqref{E:finite-determinants-volumes} combined with the definition of the volume function gives
$${\rm det }_{R[[T^{-1}]]}(1+\Delta_\xi\,|\, A_1)|_{T=t}=\frac{|A_2|_G}{|A_1|_G}=\frac{{\rm Vol}(M_2)}{{\rm Vol}(M_1)},$$
which concludes the proof of the Lemma.
\end{proof}

Now, we are ready to prove Theorem \ref{T:Volume}. Fix an $R$--linear splitting $s_1$ for $\pi_1$ and let Let $\widetilde{\Lambda_1}$ be an $A[G]$--free, $(M_1, s_1)$--admissible lattice. Let $\widetilde{\Lambda_2}$ be a free $A[G]$--lattice containing $\Lambda_2$. For $i=1,2$, let ${\bf e_i}$ be an ordered $A[G]$--basis for $\widetilde{\Lambda_i}$. Let $X\in{\rm GL}_n(\Fq((t^{-1}))[G])$ be the transition matrix between ${\bf e_1}$
and ${\bf e_2}$, i.e. ${\bf  e_2}=X\cdot {\bf  e_1}$.

In what follows, we view the matrix ring ${\rm M}_n(\Fq((t^{-1}))[G])$ endowed with its $t^{-1}$--adic topology. In this topology, ${\rm GL}_n(\Fq((t^{-1}))[G])$ is open in ${\rm M}_n(\Fq((t^{-1}))[G])$ and it has a basis of open neighborhoods of $1$ consisting of $(1+t^{-i}{\rm M}_n(\Fq[[t^{-1}]][G])_{i\geq 0}$. Also, ${\rm GL}_n(\Fq(t)[G])$ is dense in ${\rm GL}_n(\Fq((t^{-1}))[G])$. These facts imply that, if we fix an $N>a$, we can write
$$X=B\cdot X_0,\qquad   X_0\in (1+t^{-2N}{\rm M}_n(\Fq[[t^{-1}]][G]),\quad B\in{\rm GL}_n(\Fq(t)[G]).$$
Let $\phi_X, \phi_{X_0}, \phi_B :K_\infty\simeq K_\infty$ be the $\Fq((t^{-1}))[G]$--linear isomorphisms whose matrices in the basis ${\bf e}_1$ are $X$, $X_0$ and $B$, respectively. We have a commutative diagram
of morphisms in the category of compact $R[t]$--modules, with exact rows and vertical isomorphisms
\begin{equation}\label{E:push-out}\xymatrix{
0\ar[r] & K_\infty/\Lambda_1\ar[r]\ar[d]^{\phi_{X_0}}_{\wr} & M_1\ar[r]^{\pi_1}\ar[d]^{\phi}_{\wr} & H_1\ar[r]\ar[d]^{{\rm id}}_{=} & 0 \\
0\ar[r] & K_\infty/\Lambda_1'\ar[r] & M_1'\ar[r]^{\pi_1'} & H_1\ar[r] & 0,}\end{equation}
where $\Lambda_1':=\phi_{X_0}(\Lambda_1)$, $M_1':=M_1\times_{K_\infty/\Lambda_1}K_\infty/\Lambda_1'$ and $\phi$ is induced by $\phi_{X_0}$. In other words, the lower exact sequence
is the push--out along $\phi_{X_0}$ of the upper one.

Now, note that $M_1'$ is an object in class $\mathcal C$ (the lower exact sequence is its structural exact sequence). Most importantly, note
that, since $X_0\in (1+t^{-2N}{\rm M}_n(\Fq[[t^{-1}]][G])$, the $R[t]$--linear isomorphism $\phi:M_1\simeq M_1'$ is $(2N)$--tangent to the identity. Therefore,
the $R$--linear isomorphism $\gamma\circ\phi^{-1}:M_1'\simeq M_2$ is $(2N)$--tangent to the identity.  Further, since $B\in{\rm GL}_n(\Fq(t)[G])$, it is easy to see from the definitions
that $\Lambda_1'$ and $\Lambda_2$ are contained in a common $A[G]$--lattice of $K_\infty$. Consequently, Lemma \ref{L:common-over-lattice} applied to $\gamma\circ\phi^{-1}$ gives
$${\rm det}_{R[[T^{-1}]]/T^{-N}}(1+\Delta_{\gamma\circ\phi^{-1}}\,|\, M'_1)|_{T=t} = \frac{{\rm Vol}(M_2)}{{\rm Vol}(M'_1)}\mod t^{-N}.$$
However, since $\phi$ is $R[t]$--linear, we have $\phi t\phi^{-1}=t$, so $(1+\Delta_{\phi^{-1}})=1$ on $M'_1$. Therefore, the above congruence combined with \eqref{E:composition} gives
\begin{equation}\label{E:M1-prime}
{\rm det}_{R[[T^{-1}]]/T^{-N}}(1+\Delta_{\gamma}\,|\, M_1)|_{T=t} = \frac{{\rm Vol}(M_2)}{{\rm Vol}(M'_1)}\mod t^{-N}.\end{equation}

Now, let $s_1':=\phi\circ s_1$. Diagram \eqref{E:push-out} shows that $s_1'$ is a section of $\pi_1'$ and that $\widetilde{\Lambda_1'}:=\phi_{X_0}(\widetilde{\Lambda_1})$ is
an $(M_1', s_1')$--admissible lattice. Since $\phi$ gives an $R[t]$--linear isomorphism
$$\phi: (\widetilde{\Lambda_1}/\Lambda_1\times s_1(H_1))\simeq (\widetilde{\Lambda'_1}/\Lambda'_1\times s'_1(H'_1)),$$
we have an equality $|\widetilde{\Lambda_1}/\Lambda_1\times s_1(H_1)|_G=|\widetilde{\Lambda'_1}/\Lambda'_1\times s'_1(H'_1)|_G.$ Therefore, we have
$$\frac{{\rm Vol}(M_1)}{{\rm Vol}(M'_1)}=[\widetilde{\Lambda_1'}:\widetilde{\Lambda_1}]_G=[\phi_{X_0}(\widetilde{\Lambda_1}):\widetilde{\Lambda_1}]_G={\rm det}(X_0)\equiv 1\mod t^{-2N}.$$
Combined with \eqref{E:M1-prime}, this leads to
$${\rm det}_{R[[T^{-1}]]/T^{-N}}(1+\Delta_{\gamma}\,|\,M_1)|_{T=t}\equiv\frac{{\rm Vol}(M_2)}{{\rm Vol}(M_1)}\mod t^{-N}.$$
After taking a limit for $N\to\infty$, this concludes the proof of Theorem \ref{T:Volume}.
\end{proof}

\section{The main theorems}\label{S:main-theorems}
In this section, we prove the main results of this paper, announced in \S\ref{S:intro-ETNF}. We work with the notations, and under the assumptions in \S\S1.1--1.5.

\subsection{The equivariant Tamagawa number formula for Drinfeld modules}\label{S:Main-ETNF} Below, we state and prove the $G$--equivariant generalization of Taelman's class--number formula \cite{T12}.
\begin{theorem}[the ETNF for Drinfeld modules]\label{T:Main-ETNF}  If $\cM$ is a taming module for $K/F$ and $E$ is a Drinfeld module of structural morphism
$\varphi_E:\Fq[t]\to\cO_F\{\tau\}$, then we have the following equality in $(1+t^{-1}\F_q[[t^{-1}]][G])$.
$$\Theta_{K/F}^{E, \cM}(0)=\frac{{\rm Vol}(E(K_\infty)/E(\cM))}{{\rm Vol}(K_\infty/\cM)}.$$
\end{theorem}
\begin{proof} Note that $M_1:=E(K_\infty)/E(\cM)$ is an object in class $\mathcal C$ of structural exact sequence \eqref{E:exponential-sequence}, and so is $M_2:=K_\infty/\cM$. By definition,
$M_1$ and $M_2$ have identical $R$--module structures. However, while $t$ acts on $M_2$ naturally, $t$ acts on $M_1$ via the $R$--linear operator $\varphi_E(t)\in \cO_F\{\tau\}$. Consider $\gamma:={\rm id}$ as a continuous $R$--linear operator
$$\gamma: M_1\simeq M_2, \qquad \gamma(x)=x,\,\,\forall x\in M_1.$$
Since $\gamma\circ\iota_1=\widetilde{{\rm exp}_E}$ and ${\rm exp}_E:K_\infty\to K_\infty$ is given by an everywhere convergent, $R$--linear power series in $F_\infty[[z]]$ of the form ${\rm exp_E}=z+a_1z^q+a_2z^{q^2}+...$, Proposition \ref{P:Ntangpowerseries} shows that $\gamma$ is infinitely tangent to the identity. Consequently, Theorem \ref{T:Volume} shows that we have
$${\rm det}_{R[[T^{-1}]]}(1+\Delta_\gamma\,|M_1\,)|_{T=t}=\frac{{\rm Vol}(M_2)}{{\rm Vol}(M_1)}.$$
Since $\gamma={\rm id}$, if we identify $M_1$ with $K_\infty/\cM$ as $R$--modules, the $R[[T^{-1}]]$--linear operators
$$(1-\gamma^{-1}t\gamma\cdot T^{-1}), \qquad (1-t\cdot T^{-1})$$ on $M_1[[T^{-1}]]$
become $(1-t\cdot T^{-1})$ and $(1-\varphi_E(t)\cdot T^{-1})$, respectively, on $K_\infty/\cM[[T^{-1}]]$. Therefore, the last displayed equality can be rewritten
$${\rm det}_{R[[T^{-1}]]}\left(\frac{1-t\cdot T^{-1}}{1-\varphi_E(t)\cdot T^{-1}}\,\bigg|\,K_\infty/\cM\,\right )\Bigg|_{T=t}=\frac{{\rm Vol}(K_\infty/\cM)}{{\rm Vol}(E(K_\infty)/E(\cM))}.$$
Now, Corollary \ref{C:TraceFormula} identifies the left--hand side of the equality above with $\Theta_{K/F}^{E, \cM}(0)^{-1}$, which gives the desired result.
\end{proof}

\begin{corollary}\label{C:Main-G-not-p} If $p\nmid |G|$, then we have the following equality in $(1+t^{-1}\F_q[[t^{-1}]][G])$:
$$\Theta_{K/F}^E(0)=[\cO_K: {\rm exp}^{-1}_E(\cO_K)]_G\cdot |H(E/\cO_K)|_G.$$
\end{corollary}
\begin{proof}
In this case, the extension $K/F$ is tame, so all taming modules are equal to $\cO_K$. Also, all $A[G]$--lattices are $G$--c.t., therefore
$A[G]$--projective, and the same holds for the $A[G]$--module $H(E/\cO_K)$. Therefore, the exact sequence \eqref{E:exponential-sequence} (with $\cM=\cO_K$)
is split in the category of $A[G]$--modules.  So, if $s$ is an $A[G]$--linear section of $\pi$, we have equalities
$${\rm Vol}(E(K_\infty)/E(\cO_K))=\frac{|s(H(E/\cO_K)|_G}{[{\rm exp}^{-1}_E(\cO_K):\Lambda_0]_G}, \qquad {\rm Vol}(K_\infty/\cO_K)=\frac{1}{[\cO_K:\Lambda_0]_G},$$
where $\Lambda_0$ is the auxiliary $A[G]$--lattice fixed in Definition \eqref{D:volume}
Now, since we have an isomorphism $s(H(E/\cO_K))\simeq H(E/\cO_K)$ of $A[G]$--modules, the desired result follows directly from Theorem \ref{T:Main-ETNF} and the equalities above.
\end{proof}

\begin{remark}\label{R:Taelman-Angles-Taelman} As pointed out in the introduction, if $G$ is the trivial group (i.e. $K=F$), the above Corollary is precisely Taelman's class number formula \cite{T12}. If $K:=F(C[v_0])$ is the extension of $F$ obtained by adjoining the $v_0$--torsion points of the Carlitz module $C$, for some $v_0\in{\rm MSpec}(A)$, then the above Corollary applies because
$G$ is a subgroup of $(A/v_0)^\times$ and therefore of order coprime to $p$, and it implies the main result of Angles--Taelman in \cite{AT15}.
\end{remark}

\subsection{The refined Brumer--Stark conjecture for Drinfeld modules}\label{S:Main-Brumer-Stark} As an application of Theorem \ref{T:Main-ETNF}, we prove the Drinfeld module analogue of the classical refined Brumer--Stark
Conjecture for number fields.

We remind the reader that the classical refined Brumer--Stark Conjecture roughly states that the special value $\Theta_{K/F, T}(0)$ of a $G$--equivariant Artin $L$--function $\Theta_{K/F, T}:\Bbb C\to\Bbb C$, associated to an abelian extension $K/F$ of number fields of Galois group $G$, belongs to the Fitting ideal $\fitt^0_{\Z[G]}({\rm Cl}_{K,T}^{\vee})$ of the Pontrjagin dual of a certain ray--class group ${\rm Cl}_{K,T}$ of the field $K$. Here, $T$ is a certain finite set of primes in ${\rm MSpec}(\cO_F)$ and $\Theta_{K/F, T}(0)$ is a classical Artin $L$--function with some extra Euler factors at the primes in $T$. See \cite[\S 6.1]{GP15} for a precise statement and conditional proof.

This classical conjecture has tremendously far reaching applications to the arithmetic of number fields, ranging from explicit constructions of Euler Systems and of very general algebraic Hecke characters, to understanding the  $\Bbb Z[G]$--module structure of the Quillen $K$--groups $K_i(\cO_K)$. (See \cite{BG19} for more details). Its Drinfeld module analogue is the following.

\begin{theorem} [refined Brumer--Stark for Drinfeld modules]\label{T:Main-BrSt}  If $\cM$ is a taming module for $K/F$, $E$ is a Drinfeld module of structural morphism $\varphi_E:\Fq[t]\to\cO_F\{\tau\}$, and $\Lambda'$ is a $E(K_\infty)/E(\cM)$--admissible $A[G]$--lattice  in $K_\infty$ (as in \ref{D:Admissible lattice}), then we have
$$\frac{1}{[\cM:\Lambda']_G}\cdot \Theta_{K/F}^{E, \cM}(0)\in{\rm Fitt}^0_{A[G]}H(E/\cM).$$
\end{theorem}
\begin{proof} Let $s:H(E/\cM)\to E(K_\infty)/E(\cM)$ be an $R$--linear splitting for the exact sequence \eqref{E:exponential-sequence} (a right inverse for $\pi$) and let $\Lambda'$ be an $(E(K_\infty)/E(\cM), s)$--admissible lattice.
Theorem \ref{T:Main-ETNF} combined with the definition of the function ${\rm Vol}$ leads to the equality
$$\frac{1}{[\cM:\Lambda']_G}\cdot \Theta_{K/F}^{E, \cM}(0)=|\Lambda'/{\rm exp_E}^{-1}(\cM)\times s(H(E/\cM))|_G.$$
However, recall that $|M|_G$ is, by definition, a monic generator of ${\rm Fitt}^0_{A[G]}(M)$, for all finite, projective $A[G]$--modules $M$. Therefore, we have
$$\frac{1}{[\cM:\Lambda']_G}\cdot \Theta_{K/F}^{E, \cM}(0)\in {\rm Fitt}^0_{A[G]}\left(\Lambda'/{\rm exp_E}^{-1}(\cM)\times s(H(E/\cM))\right).$$
However, $\pi: (\Lambda'/{\rm exp_E}^{-1}(\cM)\times s(H(E/\cM))\twoheadrightarrow H(E/\cM)$ is a surjective, $A[G]$--linear morphism. Therefore, a basic property of Fitting ideals implies that
$${\rm Fitt}^0_{A[G]}\left(\Lambda'/{\rm exp_E}^{-1}(\cM)\times s(H(E/\cM))\right)\subseteq {\rm Fitt}^0_{A[G]}H(E/\cM)),$$
which, if combined with the last displayed statement, concludes the proof.
\end{proof}
\noindent Theorem \ref{T:Main-BrSt} has two consequences regarding the $A[G]$--module structure of Taelman's class--group $H(E/\cO_K).$
\begin{corollary} With notations as in Theorem \ref{T:Main-BrSt}, we have
$$\frac{1}{[\cM:\Lambda']_G}\cdot \Theta_{K/F}^{E, \cM}(0)\in{\rm Fitt}^0_{A[G]}H(E/\cO_K).$$
\end{corollary}
\begin{proof} This follows directly from the surjective morphism $H(E/\cM)\twoheadrightarrow H(E/\cO_K)$ of $A[G]$--modules
(see \eqref{E:class-group-surjections}), which gives an inclusion ${\rm Fitt}^0_{A[G]}H(E/\cM)\subseteq {\rm Fitt}^0_{A[G]}H(E/\cO_K).$
\end{proof}
\noindent In the case $p\nmid |G|$ we obtain a description of the full Fitting ideal of $H(E/\cO_K)$.
\begin{corollary}\label{C:Main-fullFitt}
If $p\nmid |G|$, then we have an equality of principal $A[G]$--ideals
$$\frac{1}{[\cO_K: {\rm exp}^{-1}_E(\cO_K)]_G}\Theta_{K/F}^E(0)\cdot A[G]={\rm Fitt}^0_{A[G]} H(E/\cO_K).$$
\end{corollary}
\begin{proof} This is a direct consequence of Corollary \ref{C:Main-G-not-p}.
\end{proof}
\begin{remark}\label{R:Brumer-Stark}
 Although, in general, the $L$--value $\Theta_{K/F}^{E, \cM}(0)\in (1+t^{-1}\cdot\F_q[[t^{-1}]][G])$ is transcendental over $\F_q(t)[G]$, the quotients
$$\frac{1}{[\cM:\Lambda']_G}\cdot \Theta_{K/F}^{E, \cM}$$
turn out to be elements in the integral group ring $A[G]$, for all $E(K_\infty)/E(\cM)$--admissible $\Lambda'$. One should think of $[\cM:\Lambda']_G$ as an ``integral smoothing'' period for the $L$--value  $\Theta_{K/F}^{E, \cM}(0)$.

This is in analogy with the number field situation, where the initial Artin $L$--value $\Theta_{K/F}(0)$ is in general in $\Bbb Q[G]$, but once hit with some well chosen Euler factors at primes in the finite set $T$, the ensuing $T$--modified $L$--value $\Theta_{K/F, T}(0)$ lands in the integral group ring $\Bbb Z[G]$. Also, the natural surjection $H(E/\cM)\twoheadrightarrow H(E/\cO_K)$ is in perfect analogy
with the number field surjection ${\rm Cl}_{K, T}\twoheadrightarrow{\rm Cl}_K$ from the ray--class group associated to the finite set of primes $T$ and the actual ideal--class group of $K$. (See \cite{GP15} for details.) The  only difference between Brumer--Stark for number fields and Drinfeld modules, respectively, is the fact that the former deals with Pontrjagin duals ${\rm Cl}_{K,T}^{\vee}$ of ray--class groups while the latter does not see Pontrjagin duality. This aspect is somewhat puzzling to us and requires further investigation.
\end{remark}

\section{Appendix}\label{S:appendix}

The goal of this Appendix is to develop several tools, mostly of homological nature, needed throughout the paper.

In what follows, if $S$ is a commutative ring, ${\rm Spec}(S)$ and $\mspec(S)$ denote the spectrum (set of prime ideals) and maximal spectrum (set of maximal ideals)
of $S$, respectively. If $M$ is an $S$--module and $\wp\in{\rm Spec}(S)$, then $M_{\wp}$ denotes the localization of $M$ at $\wp$, viewed as a module over the localization $S_{\wp}$ of $S$ at $\wp$. Recall that if $M$ is a finitely
generated, projective $S$--module, then $M_{\wp}$ is $S_{\wp}$--free of finite rank, denoted ${\rm rk}_{\wp}M$. The local rank function ${\rm rk}: {\rm Spec}(S)\to \Z_{\geq 0}$, $\wp\to {\rm rk}_{\wp}M$, is locally constant in the Zarisky topology of ${\rm Spec}(S)$ and therefore constant if ${\rm Spec}(S)$ is connected (i.e. if $S$ has no non--trivial idempotents.) Also, recall that a finitely generated $S$--module $M$ is projective if and only if $M_\fm$ is $S_\fm$--projective, for all $\fm\in\mspec(S)$. If $S$ is local, a theorem of Kaplansky states that $M$ is projective if and only if $M$ is free (even if $M$ is not f.g.). See \cite{Milnor} for these facts.

If $M$ is a finitely generated, projective $S$--module, $S$ is Noetherian, and $\varphi\in{\rm End}_S(M)$, there exists a unique element ${\rm det}_S(\varphi|M)\in S$ (called the determinant of $\varphi$) which maps into $({\rm det}_{S_{\wp}}(\varphi_{\wp}|M_\wp))_{\wp\in{\rm Spec}(S)}$ via the canonical embedding $S\hookrightarrow\prod_{\wp}S_\wp$. One can see that
\begin{equation}\label{E:det-projective}
{\rm det}_S(\varphi|M)={\rm det}_S(\varphi\oplus{\rm id}_Q|M\oplus Q),\end{equation}
where $Q$ is any finitely generated $S$--module such that $(M\oplus Q)$ is $S$--free. (See \cite{GP12}.)

If $M$ is a finitely presented $S$--module, then we denote by ${\rm Fitt}^0_S(M)$ the $0$--th Fitting ideal of $M$. Given an $S$--module presentation for $M$
$$S^m\overset{\theta}\longrightarrow S^n\to M\to 0,$$
then ${\rm Fitt}^0_S(M)$
 is (the ideal in $S$ equal to)  the image ${\rm Im}({\rm det}\circ\wedge^n\theta)$ of the $S$--module morphism
 $$\wedge^nS^m\overset{\wedge^n\theta}\longrightarrow\wedge^nS^n\overset{\rm det}\simeq S.$$
 See \cite{GP12} for the main properties of Fitting ideals used in this paper.

If $G$ is a finite group and $M$ is a $\Z[G]$--module, then $\widehat{\rm H}^i(G, M)$ is the $i$--th Tate cohomology group of $M$, for all $i\in\Z$. Recall that $M$ is called $G$--cohomologically trivial (abbreviated $G$--c.t. in this paper) if $\widehat{\rm H}^i(H, M)=0$, for all subgroups $H$ of $G$ and all $i\in\Z$. For the properties of Tate cohomology needed throughout, the reader can consult Ch. VI of \cite{CF67}.

If $R$ is a commutative ring and $G$ is an abelian group, then $I_G$ denotes the augmentation ideal of the group ring $R[G]$, i.e. the kernel of the $R$--algebra augmentation morphism
$s_G:R[G]\to R$, sending $g\to 1$, for all $g\in G$.

\subsection{Characteristic $p$ group--rings and their modules}\label{S:Algprelim}

In what follows $R$ is a commutative ring of characteristic $p$, $G$ is a finite, abelian group, and $M$ is an $R[G]$--module.

\begin{lemma}\label{L:Spec} If $G$ is a $p$--group, then the following hold.
\begin{enumerate}
\item There is a one--to--one correspondence, preserving
maximal ideals
$${\rm Spec}(R)\leftrightarrow{\rm Spec}(R[G]), \qquad \fp\to \fp_G:=(\fp, I_G).$$
 \item If $R$ is local (e.g. a field), then $R[G]$ is local.
 \item For all $\fp\in{\rm Spec}(R)$, we have $R[G]_{\fp_G}=R_\fp[G]$ and $M_{\fp_G}=M_{\fp}.$
\end{enumerate}
\end{lemma}
\begin{proof} (1) First, note that every element of $I_G$ is nilpotent. Indeed, if $x\in I_G$, then $x=\sum_{\sigma\in G}a_\sigma\cdot(\sigma-1)$, for some $a_\sigma\in R$. Since ${\rm char}(R)=p$ and $G$ is a $p$--group, we have
$$
x^{|G|}=\sum_{\sigma}a_\sigma^{|G|}\cdot(\sigma^{|G|} - 1) = 0
$$
It follows that $I_G$ is contained in every prime ideal of $R[G]$.  Now, (1) follows since $R[G]/R$ is an integral extension of rings and therefore any prime (maximal) ideal in $R[G]$ contains a unique prime (maximal) ideal in $R$, plus the obvious isomorphisms of rings $R[G]/{\fp_G}\simeq R/\fp$, for all $\fp$ as above and $R[G]/I_G\simeq R$.

(2) is an immediate consequence of (1).

(3) Let $\fp\in{\rm Spec}(R)$.  Note that $R_\fp[G]$ embeds in $R[G]_{\fp_G}$.  To show that the two are equal, suppose that $x\in (R[G] \setminus \fp_G)$.  This means that $s_G(x)\in (R\setminus\fp)$.  Since $G$ is a $p$-group and ${\rm char}(R)=p$, we have
$$
x^{|G|^\alpha}=s_G(x)^{|G|^\alpha},
$$
for some $\alpha\in\Z_{\gg 0}$. It follows that $x$ is invertible in $R_\fp[G]$ with inverse $\frac{x^{|G|^\alpha-1}}{s_G(x)^{|G|^\alpha}}$.  Hence, $R[G]_{\fp_G} = R_\fp[G]$. The fact that $M_{\fp}=M_{\fp_G}$ follows similarly.
\end{proof}

\begin{lemma}\label{L:gct-local} Assume that $G$ is a $p$--group.
Assume that $R$ is a DVR and $M$ is finitely generated, or that $R$ is a field and $M$ is arbitrary.
Then, the following are equivalent.
\begin{enumerate}
\item $M$ is $R[G]$--free.
\item $M$ is $R[G]$--projective.
\item $M$ is $R$--free and $G$--c.t.
\end{enumerate}
\end{lemma}
\begin{proof} Since in this case $R[G]$ is a local ring (see (2) of the previous Lemma), (1) and (2) are obviously equivalent. Now, if $R$ is a field,
then the equivalence of (2) and (3) is proved similarly to Theorem 6 in Ch. VI \S9 of \cite{CF67}. (In loc.cit. $R=\F_p$.) If $R$
is a DVR of maximal ideal $\fm=\pi R$, then the equivalence of (2) and (3) is proved similarly to Theorem 8 in Ch. VI \S9 of \cite{CF67},
by replacing $\Z$, $p$, and $\F_p$ with $R$, $\pi$, and $R/\pi$, respectively.
\end{proof}

\begin{lemma}\label{L:gct-Dedekind} If $R$ is a Dedekind domain, $G$ is a $p$--group, and $M$ is a finitely generated $R[G]$--module, then the following
are equivalent.
\begin{enumerate}
\item $M$ is $R[G]$--projective.
\item $M$ is $R$--projective and $G$--c.t.
\end{enumerate}
\end{lemma}
\begin{proof}  Since $M$ is finitely generated (f.g.), Lemma \ref{L:Spec} shows that $M$ is $R[G]$--projective iff $M_{\fm}$ is $R_{\fm}[G]$--projective for all $\fm\in\mspec(R).$ However, since $R_{\fm}$ is a DVR, Lemma \ref{L:gct-local} shows that this happens iff $M_\fm$ is $R_{\fm}$--free and $G$--c.t., for all $\fm$. Now, since $M$ is f.g. as an $R$--module as well, this happens iff $M$ is $R$--projective and $G$--c.t. Here, we have used the $R_{\fm}$--module isomorphisms $\widehat{\rm H}^i(H, M)_{\fm}\simeq \widehat{\rm H}^i(H, M_{\fm}),$
for all $i\in\Z$ and all subgroups $H$ of $G$. These are consequences of the flatness of the localization functor and the construction of Tate cohomology
via projective resolutions.
\end{proof}

Now, if $G$ is not necessarily a $p$--group, we let $G=P\times\Delta$, where $P$ is the $p$--Sylow subgroup of $G$ and $\Delta$ its complement. Assume that $R$ is a Dedekind domain.
For a character $\chi:\Delta\to \overline{Q(R)}$ with values in the separable closure of the field of fractions $Q(R)$ of $R$, we denote by ${\widehat{\chi}}$ its equivalence class under
the equivalence relation $\chi\sim\sigma\circ\chi$ given by conjugation with elements $\sigma$ in the absolute Galois group $G_{Q(R)}$. It is easily seen that the irreducible idempotents of $R[G]$ are indexed by these
equivalence classes and are given by
$$e_{{\widehat{\chi}}}:=\frac{1}{|\Delta|}\sum_{\psi\in{\widehat{\chi}}, \delta\in\Delta}\psi(\delta)\cdot\delta^{-1}, \qquad\text{ for all }{\widehat{\chi}}\in\widehat\Delta(R).$$
Here, $\widehat{\Delta}(R)$ denotes the set of all equivalence classes of characters described above. Implicitly, we have picked and fixed representatives $\chi\in{\widehat{\chi}}$, for all ${\widehat{\chi}}\in\widehat{\Delta}(R)$.
Consequently, we have ring isomorphisms
\begin{equation}\label{E:character-decomposition}
R[G]=\bigoplus_{{\widehat{\chi}}}e_{{\widehat{\chi}}}R[G]\simeq \bigoplus_{{\widehat{\chi}}}R(\chi)[P],
\end{equation}
where  $R(\chi)$ is the Dedekind domain obtained from $R$ by adjoining the values of $\chi$ and the isomorphism $e_{{\widehat{\chi}}}R[G]\simeq R(\chi)[P]$ is given by the usual $\chi$--evaluation map along $\Delta$,
for all ${\widehat{\chi}}$. For any $R[G]$--module, we have similar decompositions
$$M=\operatornamewithlimits{\bigoplus}_{{\widehat{\chi}}}e_{{\widehat{\chi}}}M\simeq\operatornamewithlimits{\bigoplus}_{{\widehat{\chi}}}M^\chi,$$
where $M^\chi:=M\otimes_{R[G]}R(\chi)[P].$

We let $I_{{\widehat{\chi}}}=:\ker(s_\chi)$, where $s_\chi$ is the following composition of $R$--algebra morphisms
$$s_\chi: R[G]\overset{\chi}{\twoheadrightarrow}R(\chi)[P]\overset{s_P}{\twoheadrightarrow}R(\chi).$$
Note that these are generalizations of the augmentation ideals $I_G$ and maps $s_G$ considered earlier. For every $\fp_\chi\in{\rm Spec(R(\chi))}$, we let $\fp_{\chi, G}:=s_\chi^{-1}(\fp_\chi).$
The following are immediate consequences of Lemma \ref{L:Spec}.
\begin{equation}
{\rm Spec}(R[G])=\bigcup^{\bullet}_{{\widehat{\chi}}}\,{\rm Spec}(R(\chi)[P]),\qquad{\rm Spec}(R(\chi)[P])=\{ \fp_{\chi, G}\, | \,\fp_\chi \in {\rm Spec}(R(\chi))\}\end{equation}
\begin{equation} R[G]_{\fp_{\chi, G}} = R(\chi)_{\fp_\chi}[P], \qquad M_{\fp_{\chi, G}}=M_{\fp_\chi}=(M^\chi)_{\fp_\chi}.\end{equation}
Note that the minimal primes (equivalently, non--maximal primes) in ${\rm Spec}(R[G])$ are the ideals $I_{{\widehat{\chi}}}$. Consequently, the connected components of ${\rm Spec}(R[G])$ are ${\rm Spec}(R(\chi)[P])$, for all ${\widehat{\chi}}$.
Further, we have the following consequence of the previous Lemmas.
\begin{corollary}\label{C:gct-arbitraryG} Let $R$ be a Dedekind domain or a field of characteristic $p$. Let $G$ be a finite, abelian group and $M$ a finitely generated $R[G]$--module. The following hold.
\begin{enumerate}[(1)]
\item $M$ is $R[G]$--projective iff $M$ is $R$--projective and $P$--c.t. iff $M$ is $R$--projective and $G$--c.t.
\item If $R$ is a DVR or a field, then $R[G]$ is a semilocal ring (i.e. a finite direct sum of local rings) of local direct summands $R(\chi)[P]$, for all $\widehat\chi$.
\item If $R$ is a DVR or a field, then $M$ is $R[G]$--free iff $M$ is $R[G]$--projective of constant rank.
\end{enumerate}
\end{corollary}
\begin{proof} This is immediate from the previous Lemmas. Please note that, in this context (where multiplication by $p$ on $M$ is the $0$--map), $P$--coh. triviality is equivalent to $G$--coh. triviality,
as $\widehat{\rm H}^i(\Delta, M)=0$ (since both $|\Delta|$ and $p$ annihilate these groups), which forces the usual restriction maps
 ${\rm res}_i:\widehat{\rm H}^i(G, M)\to \widehat{\rm H}^i(P, M)$ to be isomorphisms, for all $i$. Also, if $R$ is a field and $M$ is $R[G]$--projective of constant local rank $n$, then $M^\chi\simeq R(\chi)[P]^n$ as $R(\chi)[P]$--modules, for all $\widehat\chi$. Consequently, $M\simeq\oplus_{\widehat\chi}R(\chi)[P]^n\simeq R[G]^n$ as $R[G]$--modules.
\end{proof}

\subsection{The relevant projective modules}\label{S:projective-modules} In what follows, we work with the data $K/F$, $G$, $\F_q$, $A:=\Fq[t]$, $\infty$, $\cO_F$, $\cO_K$, $F_\infty$, $K_\infty$ and hypotheses in the Introduction. Let $n:=[F:\F_q(t)]$.

\begin{proposition}\label{P:lattices-projective} The following hold.
\begin{enumerate}
\item $K_\infty$ is a free $\F_q((t^{-1}))[G]$--module of rank $n$.
\item If $\Lambda$ is an $A[G]$--lattice in $K_\infty$, its $\Fq(t)$--span $\F_q(t)\Lambda$ is a
free $\F_q(t)[G]$--module of rank $n$.
\item For any two $A[G]$--lattices $\Lambda_1, \Lambda_2\subseteq K_\infty$ such that
$F_q(t)\Lambda_1=\F_q(t)\Lambda_2$, there exists a free $A[G]$--lattice $\Lambda\subseteq K_\infty$,
such that $\Lambda_1, \Lambda_2\subseteq\Lambda.$
\end{enumerate}
\end{proposition}
\begin{proof} (1) Hilbert's normal basis theorem asserts that $K\simeq F[G]$, as $F[G]$--modules. Consequently,
$K_\infty=K\otimes_F F_\infty\simeq F_\infty[G]$, as $F_\infty[G]$--modules. Now, since
$$F_\infty=F\otimes_{\Fq(t)}\Fq((t^{-1}))\simeq F_q((t^{-1}))^n,$$
as $\F_q((t^{-1}))$--modules, part (1) follows.

(2) Let $V=\Fq(t)\Lambda$. By the definition of $A[G]$--lattices in $K_\infty$ and part (1), we have an isomorphism and equality of $F_\infty[G]$--modules
$$V\otimes_{\Fq(t)}\F_q((t^{-1}))\simeq F_q((t^{-1})) V=K_\infty\simeq \F_q((t^{-1}))[G]^n.$$
Consequently, $V\otimes_{\Fq(t)}\F_q((t^{-1}))$ is $G$--c.t. However, since $\F_q((t^{-1}))$ is a faithfully flat $\Fq(t)$--module, for all $i\in\Z$ and $H$ subgroup of $G$ we have
$$\widehat{\rm H}^i(H, V)\otimes_{\Fq(t)}\Fq((t^{-1}))\simeq\widehat{\rm H}^i(H, V\otimes_{\Fq(t)}\Fq((t^{-1})))=0.$$
Consequently (again, faithfull flatness), $\widehat{\rm H}^i(H, V)=0$, for all $i$ and $H$. Therefore $V$ is a $G$--c.t. $\F_q(t)[G]$--module. By Corollary \ref{C:gct-arbitraryG}(1), $V$ is a  projective $\F_q(t)[G]$--module. Now, it is easily seen that the local rank function of $V$ over ${\rm Spec}(\F_q(t)[G])$
is the same as the local rank function of $V\otimes_{\Fq(t)}\Fq((t^{-1}))$ over ${\rm Spec}(\F_q((t^{-1}))[G])$. Therefore this function is constant equal to $n$. Now, (2) follows form Corollary \ref{C:gct-arbitraryG}(3).

(3) Let $V:=\Fq(t)\Lambda_1=\Fq(t)\Lambda_2$. By (2), $V$ is a free $F_q(t)[G]$--module of rank $n$. Pick a basis $\{e_1,\dots, e_n\}$ of this free module.
It is easily seen that there exists an $f\in A\setminus\{0\}$, such that $\Lambda:=A[G]\frac{e_1}{f}\oplus\dots\oplus A[G]\frac{e_n}{f}$ is a free $A[G]$--lattice containing $\Lambda_1$ and $\Lambda_2$.
\end{proof}

In what follows, for a prime $v$ of $F$, we let $F_v$ and $\cO_v$  be the completion of $F$ at $v$ and its ring of integers, respectively.
We use similar notations for primes $w$ of $K$. If $v$ is a prime in $F$, we let $K_v:=\prod_{w|v}K_w$ and $\cO_{K_v}:=\prod_{w|v}\cO_w$, where the products are taken over all the primes $w$ in $K$ sitting above $v$. We endow these products with the product of the $w$--adic topologies.
Also, $\tau$ will denote the $q$--power Frobenius endomorphism of any $\Fq$--algebra.

\medskip
The following is a classical theorem of E. Noether (see \cite{U70} and the references therein.)
\begin{theorem}\label{T:Noether} Let $\cK/\cF$ be a finite Galois extension of Galois group G. Let $\cR$ be a Dedekind domain whose field of fractions is $\cF$ and let $\cS$ be the integral closure of $\cR$ in $\cK$. Then $\cS/\cR$ is tamely ramified if and only if $\cS$ is a projective $\cR[G]$--module of constant rank $1$.
\end{theorem}

\noindent The above result justifies the following definition.

\begin{definition}\label{D:taming-module} Let $\cR$ be a Dedekind domain whose field of fractions is $F$ and let $\cS$ be its integral closure in $K$.
An $\cR\{\tau\}[G]$--submodule $\cM$ of $\cS$ is called a taming module
for $\cS/\cR$ if
\begin{enumerate}
\item $\cM$ is $\cR[G]$--projective of constant local rank $1$.
\item $\cS/\cM$ is finite and $(\cS/\cM)\otimes_{\cR}\cO_v=0$ whenever $v\in\mspec(\cR)$ is tame in $K/F$.
\end{enumerate}
\end{definition}

\begin{proposition}\label{P:taming-module} For $\cR$ and $\cS$ as in the definition above, the following hold.
\begin{enumerate}
\item Taming modules $\cM$ for $\cS/\cR$ exist.
\item If $\cS/\cR$ is tame, then any such $\cM$ equals $\cS$.
\item For any such $\cM$ and any $v\in\mspec(\cR)$, we have an $\Fq[G]$--module isomorphism
$$\cM/v\simeq \Fq[G]^{n_v}, \qquad \text{ where } n_v:=[\cR/v:\Fq].$$
\item For any such $\cM$ and $v\in\mspec(\cR)$ which is tame in $K/F$, we have $\cM/v=\cS/v$.
\end{enumerate}
\end{proposition}
\begin{proof} Let $W\subseteq \mspec(\cR)$ be the wild ramification locus for $\cS/\cR$.
For primes $v\in\mspec(\cR)$, let $S_v:=\cR\setminus v$. Let $\cR_{(v)}:=S_v^{-1}\cR$ and
$\cS_{(v)}:=S_v^{-1}\cS$. Then $\cR_{(v)}$ is a DVR and $\cS_{(v)}$ is its integral closure in $K$, which happens to
be a semilocal PID. Note that, as consequence of Theorem \ref{T:Noether} and Corollary \ref{C:gct-arbitraryG}(3), we have isomorphisms of $\cR_{(v)}[G]$--modules
$$
\cS_{(v)}\simeq \cR_{(v)}[G], \qquad \text{ for all } v\notin W.
$$
Let $\omega_0\in\cS$ be an $F[G]$--basis for $K$. Then, we can write
$$\omega_0^q=\frac{1}{f}a\cdot\omega_0,\qquad\text{ for some }f\in\cR \text{ and } a\in\cR[G].$$
Consequently, since $q\geq 2$, if we let $\omega:=f\omega_0\in\cS$, then $K=F[G]\omega$ and $\omega^q\in\cR[G]\omega.$
This shows that if, for every $v\in W$, we let
$$T_v:=\cR_{(v)}[G]\omega \subseteq\cS_{(v)},$$
then $T_v$ is an $\cR_{(v)}\{\tau\}[G]$--submodule of $\cS_{(v)}$, which is free, rank $1$ (basis $\omega$) as an $\cR_{(v)}[G]$--module.
Clearly, if $v\in W$ then the $\cR$--module $\cS_{(v)}/T_v$ is finite, torsion, supported at $v$.

(1) Let $\cM$ be the $\cR\{\tau\}[G]$--submodule
of $\cS$ fitting in the exact sequence
$$0\to\cM\to \cS\overset{j}\longrightarrow\bigoplus_{v\in W}\cS_{(v)}/T_v$$
of $\cR\{\tau\}[G]$--modules, where $j(x)=(x\mod T_v)_{v\in W}$, for $x\in\cS$. Let $\cM_{(v)}:=S_v^{-1}\cM$, for all $v\in\mspec(\cR)$. The exact seqeunce above implies that we have $\cR_{(v)}[G]$--module isomorphisms
$$\cM_{(v)}=\cS_{(v)}\simeq\cR_{(v)}[G], \text{ if }v\notin W, \qquad \cM_{(v)}=T_v\simeq\cR_{(v)}[G], \text{ if }v\in W.$$
Consequently, the $\cR[G]$--module $\cM$ is locally free of rank $1$. Therefore, $\cM$ is a projective $\cR[G]$--module of rank $1$. Now, since we also have equalities
$$(\cS/\cM)\otimes_{\cR}\cR_{(v)}=(\cS_{(v)}/\cM_{(v)})=0, \text{ if }v\notin W, \qquad (\cS/\cM)\otimes_{\cR}\cR_{(v)}=(\cS_{(v)}/T_v), \text{ if }v\in W,$$
the module $\cM$ constructed this way is a taming module for $\cS/\cR$.

(2) is a consequence of Theorem \ref{T:Noether} and condition (2) in Definition \ref{D:taming-module}.

(3) is a consequence of the $\Fq[G]$--module isomorphisms, for all $v\in\mspec(\cR)$,
$$\cM/v\simeq S_v^{-1}\cM/v\simeq \cR_{(v)}[G]/v\simeq\cR_{(v)}/v[G]\simeq\Fq[G]^{n_v}.$$

(4) is a consequence of condition (2) in Definition \ref{D:taming-module}.
\end{proof}

\begin{corollary}\label{C:taming-basis} Let $\cS/\cR$ and $\cM$ be as in Proposition \ref{P:taming-module}(1). For $v\in\mspec(\cR)$, let $\cM_v$ be the $v$--adic completion of $\cM$ and let $\pi_v\in\cR$, such that $v(\pi_v)>0$. Then $\left\{\pi_v^i\cM_v\right\}_{i\geq 0}$ is a basis of open neighborhoods of $0$ in $K_v$ consisting of free $\cO_v[G]$--modules of rank $1$.
\end{corollary}
\begin{proof}
Note that $\cO_v$ is the $v$--adic completion of $\cR$ in this case. Since $\cM$ is a f.g. $\cR$--module and $\pi_v\in\cR$, for all $v\in\mspec(\cR)$ and all $i\geq 0$ we have isomorphisms of $\cO_v[G]$--modules
$$\pi_v^i\cM_v\simeq \cM_v\simeq \cM\otimes_{\cR[G]}\cO_v[G].$$
Consequently, since $\cM$ is $\cR[G]$--projective of rank $1$, $\pi_v^i\cM_v$ is $\cO_v[G]$--projective of rank $1$. Therefore $\pi_v^i\cM_v$ is  $\cO_v[G]$--free of rank $1$ (see Corollary \ref{C:gct-arbitraryG}(3)), for all $v$ and $i$ as above.

Since the $\cO_v$--modules $\cO_{K_v}/\cM_v\simeq \cS/\cM\otimes_{\cR}\cO_v$ are finite (because $\cS/\cM$ is finite), we have
\begin{equation}\label{E:taming-open}
\pi_v^a\cO_{K_v}\subseteq\cM_v\subseteq \cO_{K_v},
\end{equation}
for $a>0$ sufficiently large. Therefore, the topological $\cO_v[G]$--modules $\left\{\pi_v^i\cM_v\right\}_{i\geq 0}$ are open in the $v$--adic topology and form a fundamental system of open neighborhoods of $0$ in $K_v$.
\end{proof}

{\bf Examples.} For us, there are two relevant examples of rings $\cR$ as above. {\bf First,} $\cR:=\cO_F$, in which case $\cS=\cO_K$. {\bf Second,} $\cR:=\cO_{F,\infty}$ which is the intersection
of all the valuation rings in $F$ corresponding to the infinite primes of $F$. This is a semilocal PID (its maximal spectrum consists of all the infinite primes in $F$) and
its integral closure $\cO_{K, \infty}$ in $K$ is described the same way in terms of the infinite primes of $K$.

\begin{definition}\label{D:infty-taming} A taming module for $\cO_K/\cO_F$ will be simply called {\it a taming module for $K/F$.} A taming module for $\cO_{K, \infty}/\cO_{F, \infty}$ will be called {\it an $\infty$--taming module for $K/F$.}
\end{definition}

\subsection{The groups of monic elements.}\label{S:monic} Let $\F$ be any finite field of characteristic $p$, let
$t$ be a transcendental element over $\F$, and let $G$ be a finite, abelian group. Write $G=P\times\Delta$,
where $P$ is the $p$--Sylow subgroup for $G$ and $\Delta$ is its complement. Note that since $G$ is finite,
we have equalities of rings $\F[G][[t^{-1}]]=\F[[t^{-1}]][G]$ and $\F[G]((t^{-1}))=\F((t^{-1}))[G]$.

\begin{definition} Define the subgroup $\F((t^{-1}))[P]^+$  of monic elements in $\F((t^{-1}))[P]^\times$ by
$$\F((t^{-1}))[P]^+:=\bigcup_{n\in\Z}\, t^n\cdot(1+t^{-1}\F[P][[t^{-1}]]).$$
\end{definition}
Now, we use the character decomposition \eqref{E:character-decomposition} for $\F[G]$ to obtain an isomorphism with a direct sum
of rings indexed with respect to $\widehat\chi\in\widehat\Delta(\F)$
$$\psi_\Delta: \F[G]((t^{-1}))\simeq \bigoplus_{\widehat\chi}\F(\chi)[P]((t^{-1})).$$
\begin{definition}
Define the subgroup $\F((t^{-1}))[G]^+$  of monic elements in $\F((t^{-1}))[G]^\times$ by
$$\F((t^{-1}))[G]^+:=\psi_{\Delta}^{-1}\left(\bigoplus_{\widehat\chi}\F(\chi)((t^{-1}))[P]^+\right).$$
\end{definition}
\begin{remark}\label{R:monic-poly} Note that a polynomial $f\in\F[G][t]$ is a monic element in the above sense, i.e.
$$f\in\F[G][t]^+:=\F[G][t]\cap\F[G]((t^{-1}))^+,$$
if and only if $\chi(f)$ is a monic polynomial in $t$ in $\F(\chi)[P][t]$ (in the usual sense), for all
$\widehat\chi\in\widehat\Delta(\F)$. Here, $\chi(f)$ is the projection of $f$ on the $\widehat\chi$--component, under the $\F[t]$-algebra isomorphism
$\F[G][t]\simeq\oplus_{\widehat\chi}\F(\chi)[P][t].$ If $R:=\oplus_i R_i$ is a finite direct sum of indecomposable commutative rings $R_i$ (i.e. ${\rm Spec}(R_i)$ connected),
then the set of monic elements in $R[t]$ is
$$R[t]^+:=\bigoplus_i R_i[t]^+,$$
where $R_i[t]^+$ are the monic polynomials in $R_i[t]$, in the usual sense. Note that the elements in $R[t]^+$ are not zero divisors in $R[t]$.
\end{remark}
\begin{proposition}\label{P:monic-dec}
For all $\F$, $t$, and $G$ as above, we have a group decomposition
$$\F((t^{-1}))[G]^\times=\F((t^{-1}))[G]^+\times \F[t][G]^\times.$$
\end{proposition}
\begin{proof}According to the last definition above, and since
$$\psi_\Delta(\F[G][t]^\times)=\bigoplus_{\widehat\chi}\F(\chi)[P][t]^\times,$$
it suffices to prove the Proposition in the case where $G=P$ is a $p$--group, which we assume below. The main ingredient needed in the proof is the following $\fm$--adic Weierstrass preparation theorem. (See \cite{L90} for a proof.)
\begin{theorem}[Weierstrass preparation] Let $(\cO, \fm)$ be an $\fm$--adically complete local ring. Let $f\in\cO[[X]]\setminus \fm[[X]]$ be a power series $f=\sum_{i\geq 0}a_iX^i$. Assume that $n\in\Bbb Z_{\geq 0}$ is minimal with the property that $a_n\not\in\fm$. Then $f$ has a unique Weierstrass decomposition
$$f = (X^n + b_{n-1}X^{n-1} + \dotsb + b_0)\cdot u,$$
with $b_i \in \fm$ and $u \in \cO[[X]]^\times$.
\end{theorem}
Note that the ring $(\F[P], I_P)$ is local (see Lemma \ref{L:Spec}.) Since the augmentation ideal $I_P$ is nilpotent (see the proof of Lemma \ref{L:Spec}),
the ring $\F[P]$ is $I_P$--adically complete. Now, let
$$g = \sum_{i \geq n} a_i t^{-i} \in \F((t^{-1}))[P]^\times, \qquad a_i\in\F[P],\quad a_n\ne 0.$$
 Let $s: \F((t^{-1}))[P] \to \F((t^{-1}))$ denote the usual augmentation $\F((t^{-1}))$--algebra morphism. Since $g$ is a unit, $s(g)=\sum_{i\ge n}s(a_i)t^{-i}$ is a unit. Therefore
 there exists a minimal $m \in \Z_{\geq n}$ such that $s(a_m) \neq 0$. This means that $m$ is minimal with the property that $a_m\notin I_P$. Now, we apply the Weierstrass preparation
 theorem to $\tilde g\in\F[P][[t^{-1}]]$, where
 $$\tilde g := t^{-n} g=\sum_{i\geq n} a_it^{-(i-n)}.$$
We get a unique Weierstrass decomposition
$$\tilde g = (b_0 + b_1 t^{-1} + \dotsb + b_{m-1}t^{-(m-1)} + t^{-m}) \cdot u, \qquad b_i\in I_P, \quad u\in \F[[t^{-1}]][P]^\times.$$
Consequently,  $u = \sum_{i\geq 0} u_it^{-i}$, with  $u_i\in\F[P]$ and $u_0\notin I_P$. In conclusion, we can write
$$
g = t^{-n} \tilde g = (u_0 b_0 t^m + u_0 b_1 t^{m-1} + \dotsb + u_0 b_{m-1}t + u_0)\cdot t^{-n-m} (1 + u_0^{-1} u_1 t^{-1} + \dotsb )
$$
Now, note that $x:= t^{-n-m} (1 + u_0^{-1} u_1 t^{-1} + \dotsb )\in \F((t^{-1}))[P]^+$, by definition. Also, note that
$y:=(u_0 b_0 t^m + u_0 b_1 t^{m-1} + \dotsb + u_0 b_{m-1}t + u_0)\in\F[t][P]^\times$. Indeed, since $s(b_i)=0$, for all $i$, and $s(u_0)\in\F^\times$, we have
$s(y)=s(u_0)\in\F[t]^\times=\F^\times.$ Therefore $y\in( \F^\times+I_P[t])$. However, Lemma \ref{L:Spec} shows that  $\F[t][P]^\times=(\F^\times+I_P[t])$. Consequently, we have written
$$g=x\cdot y, \qquad x\in \F((t^{-1}))[P]^+, \quad y\in \F[t][P]^\times.$$
The uniqueness of this writing follows from $\F((t^{-1}))[P]^+\cap \F[t][P]^\times=\{1\}$, which is obvious.
\end{proof}
\begin{corollary}\label{C:monic-rep} For $\F$, $G$ and $t$ as above, we have a canonical group isomorphism
$$\F((t^{-1}))[G]^\times/\F[t][G]^\times \simeq \F((t^{-1}))[G]^+, \qquad \widehat g \to g^+,$$
sending the class $\widehat g$ of $g\in\F((t^{-1}))[G]^\times$ to its unique monic representative $g^+.$
\end{corollary}
\begin{proof} Immediate from the group equality in the previous Proposition. \end{proof}

\subsection{Fitting ideals and their monic generators.}\label{S:fitting} In what follows, $R$ is a semilocal, Noetherian ring, i.e. $R$ is a finite direct sum $R=\oplus_i R_i$, with $R_i$ local, Noetherian ring, for all $i$. Also $M$ is an $R[t]$--module
which is finitely generated and projective as an $R$--module. With notations as in \S\ref{S:projective-modules}, the typical examples are $R=\F_q[G]$ (so $R[t]=\F_q[t][G]=A[G]$) and $M=\cM/v$, where $\cM$ is a taming module for $K/F$ and
$v\in\mspec(\cO_F)$, or $M=\Lambda_1/\Lambda_2$, with $\Lambda_1\subseteq\Lambda_2$ are projective $A[G]$--lattices in $K_\infty.$

\begin{proposition}\label{P:fittidealgen}
With notations as above, the following hold.
\begin{enumerate}
\item If $R$ is local and ${\rm rank}_RM=n$, then $\Fitt_{R[t]}^0(M)$ is principal and has a unique monic generator $|M|_{R[t]}\in R[t]^+$ which has degree $n$ and is given by
$$|M|_{R[t]}=\det\nolimits_{R[t]} (t\cdot I_n - A_t),$$
where $A_t\in M_{n}(R)$ is the matrix of the $R$--endomorphism of $M$ given by multiplication with $t$, in any $R$--basis ${\bf e}$ of $M$.
\item If $R$  is semilocal, then $\Fitt_{R[t]}^0(M)$ is principal and has a unique monic generator $|M|_{R[t]}\in R[t]^+$ given by
$|M|_{R[t]}=\sum\nolimits_i |M\otimes_R R_i|_{R_i[t]}.$
\end{enumerate}
\end{proposition}

\begin{proof} (sketch) Obviously, it suffices to prove part (1). This is a simple variation of the proof of Proposition 4.1 of \cite{GP12}. More precisely, one picks an $R$--basis ${\bf e}$ for $M$ and writes the following
sequence of $R[t]$--modules
$$0\to R[t]^n\overset{\rho_t}\longrightarrow R[t]^n\overset{\pi_t}\longrightarrow M\to 0,$$
where ${\pi_t}$ maps bijectively the standard $R[t]$--basis of $R[t]^n$ to $\bf e$ and $\rho_t$ has matrix $(t\cdot I_n-A_t)$ in the standard basis of $R[t]^n$. As in loc.cit., one proves that this sequence is exact. This yields part (1), by the definition of  $\Fitt_{R[t]}^0(M)$ and the obvious equality $R[t]^+\cap R[t]^\times=\{1\}$.
\end{proof}

\begin{definition}\label{D:G-size} If $M$ is an $A[G]$--module (i.e. an $\Fq[G][t]$--module), which is finite and $G$--c.t. (i.e. $\Fq[G]$--projective of finite rank), then
we let
$$|M|_G:=|M|_{A[G]}, $$
viewed as an element in $\in\Fq[G][t]^+=\Fq[G][t]\cap\Fq[G]((t^{-1}))^+.$
\end{definition}

\begin{lemma}\label{L:FitSequence}
Let $R$ be a semilocal, Noetherian ring. Let $A, B, C$ be $R[t]$--modules which are finitely generated and projective as $R$--modules. If we have an exact sequence of $R[t]$--modules
$$
0\rightarrow A \overset{\iota}\rightarrow B \overset{\pi}\rightarrow C \rightarrow 0,
$$
then we get the following equality of monic elements in $R[t]^+.$
$$|A|_{R[t]}\cdot |C|_{R[t]}=|B|_{R[t]}.$$
\end{lemma}
\begin{proof} (sketch) Obviously, it suffices to prove the statement when $R$ is local. Fix a section $s:C\to B$ for $\pi$ in the category of $R$--modules. Pick $R$--bases ${\bf a}$ and ${\bf c}$ for $A$ and $C$, respectively. Then, note that ${\bf b}:=\iota({\bf a})\cup s({\bf c})$ is an $R$--basis for $B$. Now, apply part (1) of the previous Proposition to compute $|A|_{R[t]}$, $|B|_{R[t]}$ and $|C|_{R[t]}$ in these respective bases.
\end{proof}

\subsection{Carlitz module Euler factors} \label{S:CarlitzEulerFactors}
The goal of this section is to prove Proposition \ref{P:EF-values}(1) and Proposition \ref{P:EF-values}(2) in the particular case where the Drinfeld module in question is the Carlitz module.
Below, we adopt the notations in Proposition \ref{P:EF-values}.

Any $\F_q[t]$--Drinfeld module $E$ defined over $\cO_F$, given by an $\F_q$--algebra morphism
$$\varphi_E: \F_q[t]\to \cO_F\{\tau\},$$
gives rise to a natural functor $M\to E(M)$ from the category of $\cO_F\{\tau\}$--modules to that of $\F_q[t]$--modules. Moreover, if $H$ is a group, then the natural $\F_q[H]$--algebra morphism
$$\varphi_E^H: \F_q[t][H]\to \mathcal \cO_F\{\tau\}[H]$$
obtained from $\varphi_E$ gives rise to a functor $M\to E^H(M)$ from the category of $\cO_F\{\tau\}[H]$--modules to that of $\F_q[t][H]$--modules. Note that the $\F_q[H]$--module structures of $M$ and $E^H(M)$ are identical, while their $\F_q[t]$--module structures are different, in general.
Also, it is immediate that if $H$ is a subgroup of $G$, then for any $\CO_F\{\tau\}[H]$--module $M$ we have an isomorphism
of $\F_q[t][G]$--modules
$$E^G\left(M\otimes_{\F_q[H]}\F_q[G]\right)\simeq E^H(M)\otimes_{\F_q[H]}\F_q[G].$$

We take a maximal ideal $v$ in $\CO_F$ and fix a maximal ideal $w$ of $\CO_K$ above $v$.  We let $v_0$ denote the maximal ideal of
$A$ sitting below $v$ and let $Nv$ denote the unique monic generator of $v_0^{f(v/v_0)}$, where $f(v/v_0)$ is the residual degree $[\cO_F/v:A/v_0]$.

\begin{proposition}\label{P:EF-Carlitz} Assume that $v$ is tamely ramified in $K/F$ and let $E$ be any Drinfeld module as above. Then the following hold.
\begin{enumerate}
\item The $\F_q[G]$--modules $\CO_K/v$ and $E^G(\CO_K/v)$ are free of rank $n_v:=[\cO_F/v:\F_q]$.
\item We have an equality
$$|\CO_K/v|_G=Nv.$$
\item If $C$ denotes the $\F_q[t]$--Carlitz module defined over $\CO_F$, then
$$|C^G(\CO_K/v)|_G=(Nv-{\bf e}_v\cdot\sigma_v),$$
where ${\bf e}_v=1/|I_v|\sum_{\sigma\in I_v}\sigma$ is the idempotent in $\F_q[G]$ corresponding to the trivial character of the inertia group $I_v$ and $\sigma_v$ is any Frobenius morphism for $v$ in $G$.
\end{enumerate}
\end{proposition}

\begin{proof}
To start, let $\rho\in\cO_w$ be a $\cO_v[G_v]$--basis for the free
$\cO_v[G_v]$--module $\cO_w$ of rank $1$. (See Theorem \ref{T:Noether}.) Then $\overline{\rho}:=(\rho\mod v)$ is a basis for the free $\cO_v/v[G_v]$--module $\cO_K/v=\cO_w/v$ of rank $1$. So, we have
$$\cO_w=\cO_v[G_v]\cdot\rho, \qquad \cO_K/v=\cO_w/v=\cO_v/v[G_v]\cdot\overline{\rho}\simeq \Fq[G_v]^{n_v}.$$
Let $\alpha_\tau\in\CO_v[G_v]$ such that $\tau(\rho)=\alpha_\tau\cdot\rho.$
For all $x=\sum_{\sigma\in G_v}x_\sigma\cdot\sigma\in \cO_v[G_v]$, we define
$$x^{(i)}:=\sum_{\sigma\in G_v}\tau^i(x_\sigma)\cdot\sigma=\sum_{\sigma\in G_v}x_\sigma^{q^i}\cdot\sigma,\qquad \text{ for all }i\in\Z_{\geq 0}.$$
Then we have the following obvious equality for all $i$ as above:
\begin{equation}\label{Frobenius}\tau^i(\rho)=(\alpha_\tau\cdot\alpha_\tau^{(1)}\cdot\dots\cdot\alpha_\tau^{(i-1)})\cdot\rho.\end{equation}

(1) This is Proposition \ref{P:taming-module}(3).

(2) We have an obvious isomorphism of $A[G]$--modules.
$$\cO_K/v\simeq \cO_w/v\otimes_{A[G_v]}A[G].$$
Since Fitting ideals commute with extension of scalars, this gives equalities
$${\rm Fitt}_{A[G]}(\CO_K/v)=\left({\rm Fitt}_{A[G_v]}(\CO_w/v)\right)\cdot A[G], \quad |\CO_K/v|_G=|\CO_w/v|_{G_v}.$$
However, since $\cO_v/v\simeq (A/v_0)^{f(v/v_0)}$, we have isomorphisms of $[G_v]$--modules
$$\CO_w/v\simeq \CO_v/v[G_v]\simeq \left(A[G_v]/v_0\right)^{f(v/v_0)}.$$
Consequently, we have the following equalities, which conclude the proof of part (2).
$${\rm Fitt}_{A[G_v]}(\CO_w/v)=v_0^{f(v/v_0)}=(Nv), \quad |\CO_K/v|_{A[G]}=|\CO_w/v|_{A[G_v]}=Nv.$$

(3) Since we also have an isomorphism of $A[G]$--modules.
$$C^{G}(\CO_K/v)\simeq C^{G_v}(\CO_w/v)\otimes_{A[G_v]}A[G],$$
we have equalities of ideals and monic elements, respectively.
\begin{eqnarray}
\nonumber
\quad {\rm Fitt}_{A[G]}C^{G}(\CO_K/v) &=& {\rm Fitt}_{A[G_v]}C^{G_v}(\CO_w/v)\cdot A[G] \\
\nonumber |C^G(\CO_K/v)|_{G} &=& |C^{G_v}(\CO_w/v)|_{G_v}.
\end{eqnarray}
According to Proposition \ref{P:fittidealgen}, the definition of $C$ and part (1), we have an equality
$$|C^{G_v}(\CO_w/v)|_{G_v}={\rm det}_{A[G_v]}(t\cdot I_{n_v}- (M_t+M_\tau)),$$
where $M_t$ and $M_\tau$ are matrices in $M_{n_v}(A[G_v])$ associated to multiplication by $t$ and action by $\tau$ on any
$\Fq[G_v]$--basis of $\CO_w/v$. Now, from part (2) we already know that
\begin{equation}\label{t-action}{\rm det}_{A[G_v]}(t\cdot I_{n_v}- M_t)= Nv.
\end{equation}

So, we need to analyze the matrix $M_\tau$. Let $K':=K^{I_v}$ be the maximal unramified extension of $F$ inside $K$. Let $w'$ be the prime in $\CO_{K'}$ sitting below $w$, and let $K'_{w'}$ and $\CO_{w'}$ be the usual completions. The isomorphism of $\cO_v/v[G_v]$--modules $\cO/w\simeq\cO_v/v[G_v]$ implies that
$${\bf e}_v(\CO_w/v)=\CO_{w'}/w', \qquad (1-{\bf e}_v)(\CO_w/v)=w/w^{e_v}.$$
Indeed, note that $O_{w'}/v\subseteq {\bf e}_v(\CO_w/v)$ and $w/w^{e_v}\cap \CO_{w'}/v=\{0\}$, then count dimensions of $\F_q$--vector spaces.
The last equalities combined with the fact that $e_v|(q^{n_v}-1)$ (as the extension $K_w/F_v$ is tame and $|\CO_v/v|=q^{n_v}$) and the equality $\tau^{n_v}=\sigma_v$ on $\CO_{w'}/w'$ give
\begin{equation}
\nonumber \tau^{n_v}={\bf e}_v\cdot\sigma_v \text{ on } {\bf e}_v(\CO_w/v),
\qquad \tau^{n_v}=0 \text{ on } (1-{\bf e}_v)(\CO_w/v)=w/w^e.
\end{equation}
In other words, we have
$$\tau^{n_v}={\bf e}_v\cdot\sigma_v \text{ on }\CO_w/v=\CO_v/v [G_v]\cdot{\overline\rho}.$$
If combined with \eqref{Frobenius}, this is equivalent to the following equality in $\CO_v/v[G_v]$
\begin{equation}\label{tau-nilpotent} \overline{\alpha_\tau}\cdot \overline{\alpha_\tau}^{(1)}\cdot \dots\cdot \overline{\alpha_\tau}^{(n_v-1)}={\bf e}_v\cdot\sigma_v,
\end{equation}
where $\overline{\alpha_\tau}$ is the image of $\alpha_\tau$ via the projection
$\CO_v[G_v]\twoheadrightarrow \CO_v/v[G_v]$.

Now, we extend scalars from $\Fq[G_v]$ to $\overline{\Fq}[G_v]$, where $\overline{\Fq}$ is the algebraic closure of $\Fq$.This will not alter the determinants in question. Below, we identify $\CO_w/v=\Fq(v)[G_v]\cdot\overline{\rho}$.
The ``Frobenius'' isomorphism of $\overline{\Fq}$--modules
$$\CO_v/v\otimes_{\Fq}\overline{\Fq}\simeq \overline{\Fq}^{n_v}, \quad y\otimes1\to(y,\, \tau(y),\, \dots,\, \tau^{n_v-1}(y))$$
leads to an isomorphism of $\overline{\Fq}[G_v]$--modules
\begin{eqnarray}
\nonumber \CO_w/v\otimes_{\F_q}\overline{\F}_q&\simeq&\overline{\F}_q[G_v]^{n_v}\\
\nonumber (x\cdot\overline{\rho})\otimes 1&\to& (x,\, x^{(1)},\,x^{(2)},\, \dots,\, x^{(n_v-1)}),
\end{eqnarray}
where $x\in \CO_v/v[G_v]$. Note that via the above isomorphism
$$\tau(x\cdot\overline\rho\otimes 1)\to (\overline{\alpha_\tau}\cdot x^{(1)},\overline{\alpha_\tau}^{(1)}\cdot x^{(2)},\, \dots,\, \overline{\alpha_\tau}^{(n_v-2)}\cdot x^{(n_v-1)},\, \overline{\alpha_\tau}^{(n_v-1)}\cdot x).$$

Let $\{e_i\}_i$ be the $\overline{\F}_q[G_v]$--basis of $\CO_w/v\otimes_{\F_q}\overline{\F}_q$ corresponding via this isomorphism to the standard $\overline{\F}_q[G_v]$--basis ($1$ in slot $i$ and $0$ outside) of $\overline{\F}_q[G_v]^{n_v}$. In basis $\{e_i\}_i$, the matrices $M_t$ and $M_\tau$ of the $t$ and $\tau$--actions on $(\CO_w/v\otimes_{\F_q}\overline{\F}_q)$ are given by
$$M_t=\begin{bmatrix}
    \beta_1 & 0 & \dots  & 0 &0\\
    0 & \beta_2 & \dots  & 0 &0\\
    \vdots & \vdots & \ddots & \vdots& \vdots \\
    0&0&\dots&\beta_{n_v-1}&0\\
    0 & 0 & \dots  & 0&\beta_{n_v}
\end{bmatrix}
\quad M_\tau=\begin{bmatrix}
    0 & \overline{\alpha_\tau} & 0 &\dots  & 0\\
    0 & 0 & \overline{\alpha_\tau}^{(1)}&\dots  & 0\\
    \vdots & \vdots & \vdots & \ddots& \vdots\\
    0&0&0&\dots&\quad\overline{\alpha_\tau}^{(n_v-2)}\\
    \overline{\alpha_\tau}^{(n_v-1)} & 0 & 0 &\dots&0\end{bmatrix},
$$
for some $\beta_i\in\overline{\F}_q[G_v].$ Consequently, \eqref{t-action} and \eqref{tau-nilpotent} imply that we have equalities
\begin{eqnarray}
\nonumber {\rm det}_{\overline{\F}_q[t][G_v]}(t\cdot I_n- (M_t+M_\tau))&=&\prod_{i=1}^{n_v}(t-\beta_i)-\prod_{i=1}^{n_v}\overline{\alpha_\tau}^{(i-1)}=\\
\nonumber &=&{\rm det}_{\overline{\F}_q[t][G_v]}(t\cdot I_{n_v}- M_t)-{\bf e}_v\cdot\sigma_v=\\
\nonumber &=&Nv-{\bf e}_v\cdot\sigma_v.
\end{eqnarray}
This concludes the proof of the Proposition.
\end{proof}

\bibliographystyle{amsplain}
\bibliography{DrinfeldSpecialLValuesBib}

\end{document}